\documentclass[11pt,reqno]{amsart} 
\usepackage{amsthm,amssymb,amsfonts,amsmath}
\usepackage{amsrefs}
\usepackage{fullpage}
\usepackage{bm, bbm, mathrsfs}
\usepackage{graphics, graphicx}
\usepackage[mathscr]{euscript}
\usepackage{stmaryrd}

\theoremstyle{plain}
\newtheorem{theorem}{Theorem}
\newtheorem{prop}[theorem]{Proposition}

\newtheorem{lemma}[theorem]{Lemma}

\theoremstyle{definition}

\theoremstyle{remark}
\newtheorem{remark}{Remark}

\newcommand{\me}{\ensuremath{\mathrm{e}}}

\newcommand{\dif}{\ensuremath{\mathrm{d}}}

\newcommand{\RR}{\ensuremath{\mathbb{R}}}

\newcommand{\ti}{\ensuremath{\tilde}}
\newcommand{\lam}{\ensuremath{\lambda}}

\DeclareMathOperator{\re}{Re}
\DeclareMathOperator{\im}{Im}
\newcommand{\tr}{\ensuremath{\top}}

\newcommand{\calZ}{\ensuremath{\mathcal{Z}}}

\newcommand{\mat}[1]{\ensuremath{\mathsf{#1}}}

\newcommand{\norm}[2]{\ensuremath{\|#1\|_{#2}}}

\newcommand{\intr}{\int_\RR}

\newcommand{\spm}{{\ensuremath{{\scriptscriptstyle\pm}}}}
\renewcommand{\sp}{{\ensuremath{{\scriptscriptstyle +}}}}
\newcommand{\sm}{{\ensuremath{{\scriptscriptstyle -}}}}

\newcommand{\beq}{\begin{equation}}
\newcommand{\eeq}{\end{equation}}
\newcommand\br{\begin{remark}}
\newcommand\er{\end{remark}}
\newcommand\bt{\begin{todo}}
\newcommand\et{\end{todo}}

\numberwithin{equation}{section}

\title{Stability of viscous detonations for Majda's model}
\author{Jeffrey Humpherys}
\address{Department of Mathematics, Brigham Young University, Provo, UT 84602}
\email{jeffh@math.byu.edu}
\thanks{J.H. was partially supported by NSF grant DMS-0847074 (CAREER)}
\author{Gregory Lyng}
\address{Department of Mathematics, University of Wyoming, Laramie, WY 82071}
\email{glyng@uwyo.edu}
\thanks{G.L. was partially supported by NSF grant DMS-0845127 (CAREER)}
\author{Kevin Zumbrun}
\address{Department of Mathematics, Indiana University, Bloomington, IN 47405}
\email{kzumbrun@indiana.edu}
\thanks{K.Z. was partially supported by NSF grant DMS-0801745}
\date{Last Updated:  \today}
\begin{document}
\begin{abstract} Using analytical and numerical Evans-function techniques, we examine the spectral stability of strong-detonation-wave solutions of Majda's scalar model for a reacting gas mixture with an Arrhenius-type ignition function.  We introduce an 
energy estimate to limit possible unstable eigenvalues to a compact region in the unstable complex half plane, and we use a numerical approximation of the Evans function to search for possible unstable eigenvalues in this region. Our results show, for the parameter values tested, that these waves are spectrally stable. Combining these numerical results with the pointwise Green function analysis of Lyng, Raoofi, Texier, \& Zumbrun [\emph{J. Differential Equations} \textbf{233} (2007), no. 2, 654--698.], we conclude that these waves are nonlinearly stable. This represents the first demonstration of nonlinear stability for detonation-wave solutions of the Majda model without a smallness assumption.  Notably, our results indicate that, for the simplified Majda model, there does not occur, either in a normal parameter range or in the limit of high activation energy, Hopf bifurcation to ``galloping'' or ``pulsating'' solutions as is observed in the full reactive Navier--Stokes equations.  This answers in the negative a question posed by Majda as to whether the scalar detonation model captures this aspect of detonation behavior. 
\end{abstract}
\maketitle

\section{Introduction}\label{sec:intro}
\subsection{Overview}\label{ssec:overview}
Majda's qualitative model for detonation \cite{M_SIAMJAM81} has served as an important test-bed for both theory and computation since its introduction in the early 1980s; the model captures many of the phenomena of the much more complicated reactive Euler and Navier--Stokes equations governing physical detonations without their full technical complexities. In this paper, we use analytical and numerical Evans-function techniques to determine the stability of strong-detonation-wave solutions of Majda's model. Our motivation is twofold. First, the calculations here provide a useful stepping stone in the process of developing numerical Evans-function techniques suitable to study the stability of detonation-wave solutions of the Navier--Stokes equations modeling a mixture of chemically reacting gases, a computationally intensive calculation that has up to now not been 
attempted.\footnote{See, e.g., the descriptions in \cites{LS,HZ_QAM,BZ_majda-znd} of the complexity of the significantly less intensive computation for the corresponding inviscid (ZND) problem.} 
Second, we address a question about the model originally posed by Majda \cite{M_book} himself. Namely, Majda asked whether or not the reduced, scalar model retains enough of the structure of the physical system to capture the complicated Hopf bifurcation/pulsating instability phenomena that occur for the full equations in the high-activation energy limit. 

Our results indicate that numerical Evans-function analysis of viscous detonation waves is feasible across essentially the full range of viscosity,  activation energy, and other parameters.  Indeed, in follow-up work \cite{BHLZ} we 
have extended
the approach pursued here to the full reactive Navier--Stokes equations, with interesting preliminary results.  Concerning behavior of Majda's model, of great interest in its own right given the large amount of attention given this model in the literature, we find that \emph{strong detonation waves appear to be universally stable,} across the entire range of activation energy and other parameters.  That is, we answer Majda's question in the negative; we show that the Majda model does not capture the instability and bifurcation phenomena of the full model.


\subsection{Background: scalar combustion models}\label{ssec:back}
Even in one space dimension, the compressible Navier--Stokes equations for a reacting gas mixture (see \eqref{eq:rns} below) are quite complex and are known to support solutions exhibiting diverse behaviors. For example, these equations admit distinct types of nonlinear combustion-wave solutions---detonations \& deflagrations---whose structural features are strikingly different. It is thus natural to seek models of reduced complexity which retain essential features of the physical system. In this vein, Majda~\cite{M_SIAMJAM81} (see also his related work with Colella \& Roytburd~\cites{CMR_SIAMJSSC86} and with Rosales~\cite{RM_SIAMJAM83}) introduced a ``qualitative model'' for gas-dynamical combustion with the aim of studying scenarios in which the nonlinear motion of the gas (e.g., shock waves) and the chemistry (chemical reactions among various species making up the gas mixture) are strongly coupled.\footnote{At around the same time Fickett \cite{F_AJP79} introduced a similar scalar combustion model. The principal difference between the model proposed by Fickett and Majda's model is the lack of a diffusive term in the former. Radulescu \& Tang \cite{RT_PRL11} have recently shown that a version of Fickett's model with a state-dependent forcing term better captures some of the nonlinear dynamics of the physical system; see \S\ref{sec:conclusions} below for a further discussion of this point.} That is, the movement of the gas exerts a significant influence on the chemical reactions and vice versa. The original formulation of the model \cite{M_SIAMJAM81} reads 
\begin{subequations}\label{eq:mm0}
\begin{align}
(u+qz)_t&+f(u)_x=Bu_{xx}\,, \label{eq:mm0-u}\\
z_t&=-k\phi(u)z\,.\label{eq:mm0-z}
\end{align}
\end{subequations}
Here and below, subscripts denote differentiation with respect to $x$ (a Lagrangian label) and to $t$ (time); the unknown function $u=u(x,t)$ is a lumped scalar variable representing various aspects of density, velocity, and temperature; the other unknown $z=z(x,t)\in[0,1]$ is the mass fraction of reactant in a simple one-step reaction scheme; the flux $f$ is a nonlinear convex function; $\phi$ is the ignition function---it turns on the reaction; and $k$, $q$, and $B$ are positive constants measuring reaction rate, heat release, and diffusivity, respectively. (In \S\ref{sec:prelim} below, we make precise statements about the nature of $f$, $\phi$, and the parameters for the version of \eqref{eq:mm0} which forms the basis of our analysis.) 
The key point is that the system \eqref{eq:mm0} is expected, on the one hand, to retain some of the features of the relevant physical equations which, in this case (one-step reaction, one space dimension, Lagrangian coordinates), are given by    
\begin{subequations}\label{eq:rns}
\begin{align}
\tau_t-u_x & = 0\,, \\
u_t+p_x&=\left(\frac{\mu u_x}{\tau}\right)_x\,,\\
E_t+(pu)_x&=\left(\frac{\mu u u_x}{\tau}+\frac{\kappa T_x}{\tau}+\frac{qdz_x}{\tau^2}\right)_x\,,\\
z_t&=-k\phi(T)z+\left(\frac{dz_x}{\tau^2}\right)_x\,.
\end{align}
\end{subequations}
A detailed description and derivation of \eqref{eq:rns} can be found, e.g., in the text of Williams~\cite{Williams}*{pp. 2--4, 604--617}. 
On the other hand, system \eqref{eq:mm0} is simple enough to provide a mathematically tractable setting in which the nonlinear fluid-chemistry interactions can be studied. Model \eqref{eq:mm0} can be thought of as playing the role for chemically reacting mixtures of gases that Burgers equation plays in the theory of (nonreacting) compressible gas dynamics.  

\begin{remark}\label{wnormk}
Majda~\cite{M_SIAMJAM81} originally motivated the system \eqref{eq:mm0}  simply by making a plausible argument that it retains the essential features of the fluid-chemistry interactions in \eqref{eq:rns}. Indeed, Gardner \cite{G_TAMS83} proved the existence of traveling-wave solutions of \eqref{eq:rns}, roughly speaking, by deforming the Navier--Stokes phase space to that of the model system \eqref{eq:mm0}. Moreover, in subsequent work with Rosales~\cite{RM_SIAMJAM83}, Majda showed how the model \eqref{eq:mm0} can be connected to the physical equations in a certain asymptotic regime. More precisely, Rosales \& Majda derived a simplified model for studying detonation waves in the ``Mach $1+\epsilon$'' regime where the  wave speed is is close to the sound speed. The model they derived takes the form\footnote{The $x$-coordinate appearing \eqref{eq:mr} is \emph{not} the standard spatial coordinate. Rather, it arises in the asymptotic analysis as a scaled space-time measurement of distance from the reaction zone; see \cites{RM_SIAMJAM83,CMR_SIAMJSSC86}.}
\begin{subequations}\label{eq:mr}
\begin{align}
u_t&+\left(\frac{u^2}{2}-Qz\right)_x=Bu_{xx}\,,\\
z_x&= K\phi(u)x\,,
\end{align}
\end{subequations}
and, with an appropriate choice of flux $f$ in \eqref{eq:mm0-u} and appropriate identifications between $q$, $k$ and $Q$, $K$, one connection between \eqref{eq:mm0} and \eqref{eq:mr} is that they support the same traveling waves \cite{RM_SIAMJAM83}. We note, however, as discussed in \cite{Z_ARMA11}, that the connection of these systems to the physical equations is in the regime of small heat release where strong detonation waves are known to be stable. In the setting of \eqref{eq:rns}, this was shown by Lyng \& Zumbrun \cite{LZ_ARMA04} by a continuity argument in the Evans-function framework together with Kawashima-type energy estimates for the physical system. More recently, Zumbrun \cite{Z_hf} has established the analogous result for the inviscid (ZND) version of \eqref{eq:rns}; his approach similarly relies on a continuity argument which is combined with a nonstandard asymptotic analysis to obtain control of high frequencies. The link to the physical system thus sheds little light on the question of existence/nonexistence
of galloping instabilities in the scalar model. In this context, we see that the connection of the scalar models to the physical equations is indeed qualitative and not through formal asymptotics. 
\end{remark}

We recall the two classical reductions of \eqref{eq:rns} in the standard formulation of combustion theory. These are the Chapman--Jouguet~(CJ) model (an inviscid model which features an infinitely fast reaction) and the Zeldovich--von Neumann--D\"oring~(ZND) model (likewise inviscid but featuring a finite reaction rate). Because the mathematical theory linking the three standard (CJ, ZND, Navier--Stokes) models is incomplete, scalar models like \eqref{eq:mm0} are especially attractive because they provide a tractable starting point for the development of a complete mathematical framework for understanding the equations which incorporate combustion processes with compressible gas dynamics. For example, Levy~\cite{L_CPDE92} used a vanishing viscosity (ZND limit) method to prove existence, uniqueness and continuous dependence on initial data for the initial-value problem for \eqref{eq:mm0} with $B=0$. In addition, in a mathematical study of the transition from CJ to ZND theories, Li \& Zhang~\cite{LZ_SIAMJMA02} have analyzed the infinite-reaction-rate (CJ) limit of the ZND ($B=0$) version of \eqref{eq:mm0}. There are other examples, e.g., \cites{YT_ATA84,LZ_ARMA91,ST_JDE94}. Finally, such scalar models also serve a valuable role as testbeds for the development of numerical methods. Colella, Majda, \& Roytburd~\cite{CMR_SIAMJSSC86} used \eqref{eq:mr} as a basis for the development of appropriate numerical schemes for the Navier--Stokes equations. Indeed, in these problems, the spatial scales relevant for the fluid motion and the reaction may be widely disparate; this can lead to numerical stiffness. It is thus of interest to develop codes which are capable of dealing with these separated scales. Simple scalar models can serve as valuable test cases in this development.  See also, e.g., \cites{Y_MC04,ZY_JCM05}. 

Our principal interest here is in the stability of strong detonation-wave solutions of Majda's scalar combustion model with an Arrhenius-type ignition 
function.\footnote{We discuss the extensions of the ideas and techniques of this paper to the interesting cases of weak detonations and other kinds of ignition functions, notably the ``bump'' ignition function of Lyng \& Zumbrun \cite{LZ_PD04}, in \S\ref{sec:conclusions} below.} 
As noted above, this analysis is partially motivated by a question posed by Majda in his monograph \cite{M_book}. Namely, he asked whether his reduced, scalar model retains enough of the structure of the physical system to capture the complicated Hopf bifurcation/pulsating instability phenomena that occur for the full equations in the high-activation energy limit.  Indeed, the stability properties of these waves speak to the quality of the Majda model as a reduced model for \eqref{eq:rns}. Here we find, for the parameter ranges in our study, that these waves are spectrally stable (and hence nonlinearly stable, see \S\ref{ssec:stability_evans}). Given the expectation of \emph{instability} for the physical system \eqref{eq:rns}, this finding provides a concrete demonstration of the limitations of the simplified model.

We note that verification of stability requires an \emph{all-parameters}
study, adding an additional layer of complexity beyond 
inclusion of second-order transport effects as compared to
previous normal-modes stability analyses 
(e.g., \cite{LS}) concerning the ZND, or reacting Euler equations,
in which the focus is typically on \emph{instability} in certain specific
parameter regimes.
In this regard, it seems interesting to mention a somewhat surprising 
phenomenon special to the viscous case, connected with
the small-activation energy limit $\mathcal{E}_A\to 0$ that in
the ZND context is so simple as to be usually disregarded.
This has to do with the ``cold-boundary phenomenon,'' or
low-temperature cutoff that is imposed on the ignition function
in order to permit traveling waves, which, in the small-activation
energy limit, turns out to dominate the profile existence problem.
This point is discussed further in Appendix \ref{sec:zeroe}.

%

\begin{remark}
We note that Barker \& Zumbrun \cite{BZ_majda-znd} have carried out an analogous study for the inviscid problem ($B=0$ in \eqref{eq:mm0}) and found no instability, either for the Arrhenius-type ignition function considered here or the alternative bump-type ignition function proposed in \cite{LZ_PD04}.  One might imagine that the additional degrees of freedom added by viscosity might allow for an instability. 
Indeed, Zumbrun \cite{Z_ARMA11} has shown that viscous detonation stability \emph{for all values of viscosity} implies inviscid (ZND) stability through a singular small-viscosity limit; thus, viscous stability over a range of viscosity parameters, as we investigate here, is a theoretically stronger condition than inviscid (ZND) stability.
However, our results indicate that despite these additional degrees of freedom, detonations of the viscous Majda model, like those of the inviscid model, \emph{are universally stable}.
\end{remark}

\subsection{Traveling waves \& stability}
\subsubsection{Strong \& weak detonations}
The main result of Majda's original analysis \cite{M_SIAMJAM81} is a proof of the existence of traveling-wave solutions---strong and weak detonations---for \eqref{eq:mm0}. These waves are compressive combustion waves which are analogous to the corresponding waves in standard combustion theory. Notably, some of these waves are nonmonotone with a ``combustion spike'' in qualitative agreement with classical combustion theory \cite{CourantFriedrichs}. Various authors have extended the model and/or established the existence of traveling-wave solutions in a variety of nearby scenarios. For example, Larrouturou~\cite{L_NA85} proved the existence of strong and weak detonations with an additional term $Dz_{xx}$ representing diffusion in the mass fraction equation \eqref{eq:mm0-z}, namely,
\begin{subequations}\label{eq:mm-d}
\begin{align}
(u+qz)_t+f(u)_x& =Bu_{xx}\,,\label{eq:mmd-u} \\
z_t & = Dz_{xx}-k\phi(u)z\,.\label{eq:mmd-z}
\end{align}
\end{subequations}
Again, some of these waves feature a nonmonotone spike in the profile. Logan \& Dunbar~\cite{LD_IMAJAM92} established the existence of traveling-wave solutions in a version of \eqref{eq:mm0} expanded to include reversible chemical reactions.  
Later, Razani~\cite{R_JMAA04} studied the existence of the Chapman-Jouguet detonation, a particular compressive traveling-wave solution in which the wave speed is the minimum possible, in \eqref{eq:mm-d}, and Lyng \& Zumbrun~\cite{LZ_PD04} found deflagrations---expansive combustion waves---in a version of \eqref{eq:mm0} with a modified ``bump-type'' ignition function $\phi$.
\subsubsection{Stability}
Our primary interest is in the stability of these traveling-wave solutions. It is well known that detonation waves have sensitive stability properties. For example, experiments often reveal that detonation waves have a complicated structure, typically featuring transverse wave structures and high-pressure zones, that is at odds with the classical ZND description of such waves \cite{FickettDavis}. Thus, it is of considerable interest to understand the stability properties of the basic wave solutions of combustion models. Below, we briefly survey the known stability results for the Majda model.

\subsubsection{Stability: weighted norms \& energy estimates}
A number of stability results have been obtained directly by various combinations of energy estimates, spectral analysis, and weighted norms. For example, Liu \& Ying~\cite{LY_SIAMJMA95} established, via energy estimates, the nonlinear stability of strong-detonation solutions of \eqref{eq:mm0} for sufficiently small heat release $q$; this analysis was later refined by Ying, Yang, \& Zhu~\cite{YYZ_JJIAM99}. Additionally, Ying, Yang, \& Zhu~\cite{YYZ_JDE99} extended the small-$q$ nonlinear stability result for strong detonations to \eqref{eq:mm-d}.  Using a combination of spectral analysis and Sattinger's technique \cite{S_AM76} of weighted norms, Li, Liu, \& Tan~\cite{LLT_JMAA96}  proved the nonlinear stability  of strong detonations for \eqref{eq:mm-d} under the assumption of a sufficiently small reaction rate. Similarly, Roquejoffre \& Vila~\cite{RV_AA98} studied the spectral stability of strong detonations of arbitrary amplitude in the ZND (vanishing viscosity) limit. Their result can be combined with a weighted norm argument to obtain a nonlinear result. We observe that all of the aforementioned results require, in some fashion, a smallness condition on the wave and/or model parameters.  
Outside of the Evans-function framework described below, we know of \emph{no} stability results for weak-detonation solutions of \eqref{eq:mm0} or \eqref{eq:mm-d}.  There are, however, results for weak-detonation solutions of the closely related Rosales-Majda model \eqref{eq:mr}. Roughly, Liu \& Yu~\cite{LY_CMP99} proved a nonlinear stability result for small perturbations of large waves while Szepessy~\cite{S_CMP99} established such a result for large perturbations of small waves. Notably, Szepessy's result is restricted to the Burgers nonlinearity $f(u)=u^2/2$ as his technique utilizes the Hopf--Cole transform to handle the nonlinearity.

\subsubsection{Stability: Evans function}\label{ssec:stability_evans}
We denote by $L$ the linear operator obtained by linearizing about the traveling wave in question. The Evans function, denoted by $E$, associated with $L$ is an analytic function of frequencies $\lambda$ with $\re\lambda\geq 0$ whose zeros correspond to eigenvalues of $L$.  Initiating an Evans-function program for detonation waves, Lyng \& Zumbrun~\cite{LZ_PD04} obtained partial spectral information for strong- \emph{and} weak-detonation solutions of \eqref{eq:mm0} without any restriction on the size of the wave or parameter values. In particular, their stability index calculations reveal a restriction on the nature of any possible transition to instability.  Indeed, it suggests the possibility of the ``galloping instability''---a well-known phenomenon in propagating combustion waves \cite{FickettDavis}.\footnote{The galloping instability has been investigated in the setting of the Majda model and beyond by Texier \& Zumbrun~\cites{TZ_MAA05,TZ_PD08,TZ_ARMA08,TZ_gallop}.}
 In addition, by a careful analysis of the relevant Evans functions, Zumbrun \cite{Z_ARMA11} has shown that stability of strong-detonation solutions of \eqref{eq:rns} in the small viscosity limit is equivalent to stability of the limiting (zero viscosity) ZND detonation together with stability of the viscous (purely gas-dynamical) profile for the associated Neumann shock. Using Zumbrun's result,  Jung \& Yao \cite{JY_QAM12} have extended, subject to restrictions on the nature of the ignition function $\phi$, the spectral stability result of Roquejoffre \& Vila \cite{RV_AA98} which was valid for \eqref{eq:mm0} to the viscous model considered here (see \eqref{eq:mm} below). Here, our interest in the Evans function---associated with the linearization $L$---is based on the fact that the spectral information encoded in the zeros of $E$ can be used to draw conclusions about the \emph{nonlinear} stability of the wave in question. 
\begin{prop}[Lyng-Raoofi-Texier-Zumbrun \cite{LRTZ_JDE07}]\label{prop:lrtz}Under the Evans-function condition 
\begin{equation}
\label{eq:evans_condition}
\text{$E(\cdot)$ has precisely one zero in $\{\re\lambda\geq 0\}$ (necessarily at $\lambda=0$)\,,}
\tag{$\star$}
\end{equation}
a strong or weak detonation wave of \eqref{eq:mm} is $\hat L^\infty\to L^p$ nonlinearly phase-asymptotically orbitally stable, for $p>1$. Here, 
\begin{equation}
\label{eq:l_infty_hat}
\hat L^\infty(\mathbb{R}):=\{f\in\mathscr{S}'(\mathbb{R})\;:\;(1+|\cdot|)^{3/2}f(\cdot)\in L^\infty(\mathbb{R})\}.
\end{equation}
\end{prop}

\br
We recall that if $X$ and $Y$ are Banach spaces, a traveling-wave solution is said to be $X\to Y$ \emph{nonlinearly orbitally stable} if, given initial data close in $X$ to the wave, there is a phase shift $\delta=\delta(t)$ such that the perturbed solution approaches the $\delta$-shifted profile in $Y$ as $t\to \infty$. If $\delta(t)$ converges to a limiting value $\delta(+\infty)$, we say that the wave is \emph{nonlinearly phase-asymptotically orbitally stable}.
\er

\noindent
In light of Proposition~\ref{prop:lrtz}, our primary purpose here is to locate the zeros (if any) of the Evans function. In this paper we restrict our attention to the case of strong detonations; we plan to treat the case of weak detonations in a future work. Because, in all but the most trivial cases, the Evans function is typically too complicated to be computed analytically, our approach is based on numerical approximation of the Evans function.

\subsection{Outline}
In \S\ref{sec:prelim} we describe the Majda model and briefly review the existence  problem for strong and weak detonations.  In \S\ref{sec:spectral}, we outline the spectral stability problem, the construction of the Evans function, and our algorithm for approximating the Evans function and locating its zeros. We also 
bound by means of a new energy estimate the size of the region of the 
complex plane in which one may find unstable eigenvalues. 
The final sections, \S\ref{sec:experiments} and \S\ref{sec:conclusions}, contain descriptions, results, and a discussion of our numerical calculations. Appendix \ref{sec:zeroe} contains a description of the the connection and stability problems in the limit of zero activation energy. This limiting scenario forms one of the ``organizing centers'' for the problem; see the discussion in \S\ref{sec:conclusions}.


\section{Preliminaries}\label{sec:prelim}
\subsection{Model}\label{ssec:model}
\subsubsection{Basic Assumptions}
Following Lyng, Raoofi, Texier \& Zumbrun \cite{LRTZ_JDE07}, we begin with the following version of the Majda model: 
\begin{subequations}\label{eq:mm}
\begin{align}
u_t+f(u)_x& =Bu_{xx}+qk\phi(u)z\,,\label{eq:mm1} \\
z_t & = Dz_{xx}-k\phi(u)z\,.\label{eq:mm2}
\end{align}
\end{subequations}
 Here, the real-valued unknown $u$ is an aggregated variable to be thought of as representing variously density, velocity, and/or temperature.  The unknown $z\in[0,1]$ is the mass fraction of reactant. The reaction constants, both positive, are $q$---the heat release parameter and $k$---the reaction rate. Here, $q>0$ corresponds to an exothermic reaction.
The diffusion coefficients $B$ and $D$ are also assumed to be positive constants. 
General assumptions, following \cite{M_SIAMJAM81}, are that $f\in C^2(\RR)$ with
\begin{equation}\label{eq:f-u}
\frac{\dif f}{\dif u}>0\,,\quad \frac{\dif^2 f}{\dif u^2}>0\,.
\end{equation}
%
%
\br\label{rem:flux}
Monotonicity and convexity of the flux $f$ are required to guarantee that the model reproduces, qualitatively, the the solution structure of the gas equations; see Figure \ref{fig:cj_diagram1}.
For our numerical investigations, we use the Burgers flux
\beq\label{eq:burgers}
f(u)=\frac{u^2}{2}\,,
\eeq
and, in keeping with the the restrictions imposed by \eqref{eq:f-u}, we shall restrict our attention to states in the ``physical'' domain $u\geq 0$. 
\er
Finally, we assume that 
the ignition function $\phi$ is a smooth step-type function. In general, it is assumed that $\phi$ satisfies
\beq
\phi(u)=0\;\text{for}\; u\leq u_\mathrm{ig}\,,\,\quad \phi(u)>0\; \text{for} \;u>u_\mathrm{ig}\,,\,\quad
\phi\in C^1.
\label{eq:phi}
\eeq
Here $u_\mathrm{ig}$ is a fixed lower ignition threshold. Thus, we are assuming ignition-temperature kinetics. For our numerical computations, we define $\phi$ by 
\beq\label{eq:experimentalphi}
\phi(u)=\begin{cases}
0, &\text{if}\; u\leq u_\mathrm{ig}\,, \\
\me^{-\mathcal{E}_A/(u-u_\mathrm{ig})}, &\text{if}\; u>u_\mathrm{ig}\,,
\end{cases}
\eeq
where the positive parameter $\mathcal{E}_A$ is the activation energy. 

\br\label{rem:eos}
\noindent
\begin{enumerate}
\item[(a)]
A natural generalization of \eqref{eq:mm1} is to allow the flux function $f$ to depend also on the mass fraction $z$. The interpretation of such a $z$-dependent flux is that the ``equation of state'' (which specifies the physical properties of the gas) depends on the ratio of reactant to product in the gas mixture. In the physical setting of the compressible Navier--Stokes equations for a reacting mixture equipped with a one-step chemical reaction scheme, 
\[
R\to P,
\]
which converts polytropic ideal reactant to polytropic ideal product, it is straightforward to see that conservation of mass implies that the gas constants for reactant and product must be the same \cite{Williams}. On the other hand, if $z$ is thought of as representing the progress of a large system of reactions, the effective gas constants of the reactant and product mixtures may very well be different. In this case, the equations of state modeling the total mixture should depend on $z$.  The papers of Chen, Hoff \& Trivisa \cite{CHT_ARMA03}, Lyng \& Zumbrun \cite{LZ_PD04}, and Rosales \& Majda \cite{RM_SIAMJAM83} contain related discussions. Although we do not pursue this here, this is a potential direction for future study.
\item[(b)] We note that Larrouturou's \cite{L_NA85} dissipative extension of the Majda model, equation \eqref{eq:mm-d},  does not agree with \eqref{eq:mm}. However, one reason to prefer the system \eqref{eq:mm} is that its vectorial version encompasses the artificial viscosity version of the Navier--Stokes equations for a reacting fluid mixture; see \cite{LRTZ_JDE07}.
\end{enumerate}
\er

\subsection{Connections: the profile existence problem}
\label{ssec:connections}
\subsubsection{Basic Analysis}
\label{sssec:basic}
Our interest is in traveling-wave solutions, or viscous profiles, of \eqref{eq:mm}, i.e., solutions of the form 
\beq\label{eq:twa}
u(x,t)=\hat u(x-st),\quad z(x,t)=\hat z(x-st),\quad s>0,
\eeq
which connect an unburned state $(u_\sp,z_\sp)=(u_\sp,1)$ to a completely burned state $(u_\sm,z_\sm)=(u_\sm,0)$. These are combustion waves which move from left to right leaving completely burned gas in their wake. 
Thus, we find that the traveling-wave ansatz \eqref{eq:twa} leads, after an integration, from \eqref{eq:mm} to the system of ordinary differential equations (ODEs),
\begin{subequations}\label{eq:tw}
\begin{align}
\hat u'&=B^{-1}\big(f(\hat u)-f(u_\sm)-s(\hat u-u_\sm)-q(s\hat z+D\hat y)\big)\,, \label{eq:tw1}\\
\hat z'&=\hat y\,,\label{eq:tw2}\\
\hat y'&=D^{-1}\big(-s\hat y+k\phi(\hat u)\hat z\big)\,, \label{eq:tw3}
\end{align}
\end{subequations}
where we have used $\hat y:=\hat z'$ to write the system in first order and $'$ denotes differentiation with respect to the variable $\xi:=x-st$.
We assume that the end states are such that 
\begin{equation}
u_\mathrm{ig}<u_\sm
\label{eq:u_minus}
\end{equation}
and 
\beq
u_\sp<u_\mathrm{ig}
\label{eq:u_plus}
\eeq
so that \eqref{eq:phi} implies
\begin{equation}
\phi(u_\sm)>0,\quad \phi(u_\sp)=0,\quad \phi'(u_\sp)=0\,.
\label{eq:phi_endstates}
\end{equation}
Equation~\eqref{eq:u_plus} has the physical interpretation that the unburned end state is outside of the support of the ignition function so that there is no chemical reaction on the unburnt side. 

A necessary condition for the existence of a connection is that the end states $(u_\spm,z_\spm)$ be equilibria of the traveling-wave equation \eqref{eq:tw}. This leads to the Rankine-Hugoniot condition
\begin{equation}
f(u_\sp)-f(u_\sm)=sq+s(u_\sp-u_\sm)\,,
\label{eq:rh}
\tag{RH}
\end{equation}
together with the requirements that $y_\spm=0$ and
\begin{equation}
k\phi(u_\spm)z_\spm=0\,.
\label{eq:addl_cond}
\end{equation}
Since $z_\sm=0$ and $\phi(u_\sp)=0$, it is easy to see that \eqref{eq:addl_cond} is satisfied. 
We write $a_\spm:=f'(u_\spm)$. 
If $u_\sp<u_\sm$, the combustion wave is a \emph{detonation}. 
In this case, the traveling-wave profile is said to be a \emph{strong detonation} if 
\begin{equation}
a_\sm>s>a_\sp.
\label{eq:strongdet}
\end{equation}
It is said to be a \emph{weak detonation} if 
\begin{equation}
s> a_\sm, \, a_\sp.
\label{eq:weakdet}
\end{equation}

\br\label{rem:deflagrations}
In the boundary case,  $a_\sm=s>a_\sp$, the wave is known as
a \emph{Chapman-Jouguet} detonation. We do not consider these waves here. Henceforth, we use the notations $u_\sm^\mathrm{w}$, $u_\sm^\mathrm{cj}$, and $u_\sm^\mathrm{s}$ when we want to distinguish the possible burned end states, and $u_\sm$ without superscript stands for any burned end state at $-\infty$. ``Expansive'' traveling waves with $u_\sp>u_\sm$ are called \emph{deflagrations}. In the original formulation of the Majda model \cite{M_SIAMJAM81}, no such waves exist. However, Lyng \& Zumbrun \cite{LZ_PD04} showed that, if the ignition function is suitably modified, such waves may exist. We do not consider deflagrations here. 
\er
\begin{figure}[ht] 
   \centering
   \includegraphics[width=7cm]{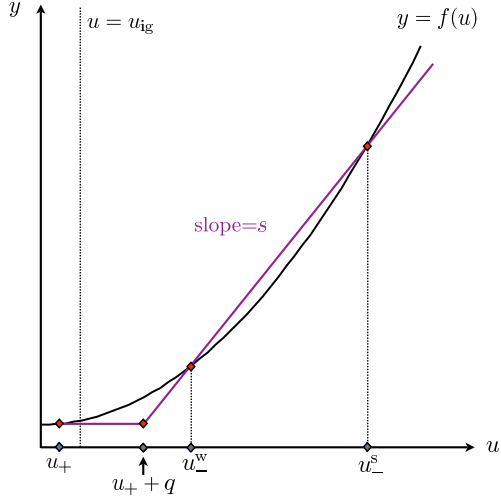}    \caption{The CJ diagram.}
   \label{fig:cj_diagram1}
\end{figure}

The following proposition describes the solutions of \eqref{eq:rh}.
\begin{prop}[\cites{M_SIAMJAM81,LRTZ_JDE07}]\label{prop:RHexist}

For fixed $u_\sp$, there exists a speed $s^\mathrm{cj}(u_\sp)$ such that
\begin{itemize}
\item for $s>s^\mathrm{cj}$ there
exist two states $u_\sm>u_\sp$ for which \eqref{eq:rh}
(but not necessarily \eqref{eq:phi_endstates})
is satisfied (weak and strong detonation),
\item for $s=s^\mathrm{cj}$ there exists one solution $u_\sm^\mathrm{cj}$ (Chapman--Jouguet detonation), and
\item for $s< s^\mathrm{cj}$, there exist no solutions $u_\sm>u_\sp$. 
\end{itemize}
See \textsc{Figure}~\ref{fig:cj_diagram1}.
\end{prop}


We assume, for the moment, that \eqref{eq:strongdet} holds. 
Linearizing \eqref{eq:tw} around the state $(u_\sm^\mathrm{s},z_\sm,y_\sm)=(u_\sm^\mathrm{s},0,0)$, we find the constant-coefficient system of ordinary differential equations 
\begin{equation}
\begin{bmatrix} \hat u \\ \hat z \\ \hat y \end{bmatrix}'=
\begin{bmatrix} 
B^{-1}(a_\sm-s) & B^{-1}(-sq) & -qB^{-1}D \\
0 & 0 & 1 \\ 
0 & kD^{-1}\phi(u_\sm) & -sD^{-1}
\end{bmatrix}\begin{bmatrix} \hat u \\ \hat z \\ \hat y \end{bmatrix}.
\label{eq:minus_lin_twode}
\end{equation}
Taking advantage of the simple block triangular structure of the coefficient matrix in \eqref{eq:minus_lin_twode}, one easily sees that it has two positive eigenvalues and one negative eigenvalue. Thus, there is a two-dimensional unstable manifold at $(u^\mathrm{s}_\sm,0, 0)$. Similarly, we note that the linearization of \eqref{eq:tw1}--\eqref{eq:tw3} about the rest point $(u_\sp,z_\sp,y_\sp)=(u_\sp,1,0)$ is 
\begin{equation}
\begin{bmatrix} \hat u \\ \hat z \\ \hat y \end{bmatrix}'=
\begin{bmatrix} 
B^{-1}(a_\sp-s) & B^{-1}(-sq) & -qB^{-1}D \\
0 & 0 & 1 \\ 
0 & 0 & -sD^{-1}
\end{bmatrix}\begin{bmatrix} \hat u \\ \hat z \\ \hat y \end{bmatrix}.
\label{eq:plus_lin_twode}
\end{equation}
Again, by virtue of the block-triangular structure, it is easy to see that there are two negative eigenvalues  and one zero eigenvalue. It is also immediate that the center manifold is a line of equilibria; no orbit may approach the rest point $(u_\sp,1,0)$ along the center manifold. Thus, for a heteroclinic connection to $(u_\sm^\mathrm{s},0,0)$, the important structure is the two-dimensional stable manifold at $(u_\sp,1, 0)$. Counting dimensions, we see that a strong-detonation connection corresponds to the structurally stable intersection of a pair of two-dimensional manifolds in $\mathbb{R}^3$. 

\br[Weak Detonations]\label{rem:weak}
Although it is not our principal interest here, we note that in the case that \eqref{eq:weakdet} holds, i.e., the wave is a weak detonation, we may repeat the above analysis to see that 
 at $(u^\mathrm{w}_\sm,z_\sm,y_\sm)=(u^\mathrm{w}_\sm,0,0)$, there is a one-dimensional unstable manifold while at $(u_\sp,z_\sp,y_\sp)=(u_\sp,1,0)$, there is a two-dimensional stable manifold and a line of equilibria (center manifold). Since no trajectory can approach the unburned state along the center manifold, a connection corresponding to a weak detonation corresponds to the intersection of the one-dimensional unstable manifold exiting the burned  end state with the two-dimensional stable manifold entering the unburned state in the the phase space $\mathbb{R}^3$. We observe that this dimensional count implies that weak detonations are isolated; weak detonations are a phenomenon of codimension one greater than than strong detonations.
We also note that while the strong detonation condition \eqref{eq:strongdet} implies that strong detonations are analogous to Lax shocks, weak detonations---i.e., waves that satisfy \eqref{eq:weakdet}---are undercompressive. For further discussion of this point, see \cite{LRTZ_JDE07}. 
\er

The following lemma is an immediate consequence of the above discussion;  exponential decay will be used below in the construction of the Evans function.   
\begin{lemma}[\cite{LRTZ_JDE07}]\label{lem:expdecay}
Traveling-wave profiles $(\hat u, \hat z, \hat y)$ corresponding to weak or strong
detonations satisfy for some $C,\theta>0$
\beq
\label{eq:expdecay}
\left|(\dif/\dif x)^k \Big((\hat u,\hat z,\hat y)(\xi)- (u,z,y)_\spm\Big)\right|\le Ce^{-\theta |\xi|},
\qquad
\xi\gtrless 0, \quad 0\le k\le 3\,.
\eeq
\end{lemma}

\subsubsection{End states and parametrization}\label{ssec:end}
Suppose $( \hat u,\hat z)$ is a traveling-wave profile of \eqref{eq:mm} satisfying \eqref{eq:strongdet}. From this point forward, for concreteness, we restrict our attention to the case of the Burgers flux $f(u)=u^2/2$ that is the basis for our numerical computations. Evidently, $(\hat u,\hat z)$ is a steady solution of 
\begin{subequations}\label{eq:mmm}
\begin{align}
u_t-su_x+(u^2/2)_x& =Bu_{xx}+qk\phi(u)z\,,\label{eq:mmm1} \\
z_t -sz_x& = Dz_{xx}-k\phi(u)z\,.\label{eq:mmm2}
\end{align}
\end{subequations}
As a preliminary step, we rescale space and time via 
\beq\label{eq:scale}
\tilde{x} = \dfrac{s}{B} x\,,\quad\tilde{t} = \dfrac{s^2}{B}t\,;
\eeq
and we rescale the dependent variables so that 
\beq\label{eq:scale2}
s\tilde u(\tilde t,\tilde x)=u(t,x)\quad\text{and}\quad\tilde z(\tilde t,\tilde x) = z(t,x)\,.
\eeq
Then \eqref{eq:mmm} takes the form
\begin{align*}
\tilde{u}_{\tilde{t}} - \tilde{u}_{\tilde{x}} + \left(\tilde{u}^2/2\right)_{\tilde{x}} &= \tilde{u}_{\tilde{x} \tilde{x}} + \tilde{q} \tilde{k} \tilde{\phi}(\tilde{u}) \tilde{z}\,,\\
\tilde{z}_{\tilde{t}} - \tilde{z}_{\tilde{x}} &= \tilde{D} \tilde{z}_{\tilde{x} \tilde{x}} - \tilde{k} \tilde{\phi}(\tilde{u}) \tilde{z}\,,
\end{align*}
where $\tilde{k} = kB/s^2$, $\tilde{\phi}(\tilde{u}) = \phi\left(\tilde{u}/s\right)$, $\tilde{q} = q/s$, and $\tilde{D} = D/B$.  Suppressing the tildes, we arrive at the system
\begin{subequations}
\label{eq:majda3}
\begin{align}
u_t - u_x + \left(u^2/2\right)_x &= u_{x x} + q k \phi(u) z\,,\label{majda3:a}\\
z_t - z_x &= D z_{x x} - k \phi(u) z\,. \label{majda3:b}
\end{align}
\end{subequations}
This calculation shows that we may take the wave speed to satisfy $s=1$ and the viscosity coefficient to be $B=1$.
Thus, \eqref{eq:rh} reduces immediately to 
\beq\label{eq:uminusq}
\frac{1}{2}(u_\sp^2-u_\sm^2)=u_\sp-u_\sm+q=0\,,
\eeq 
from which we see that the burned state $u_\sm$ is given in terms of $q$ and $u_\sp$ (strong detonation) by 
\beq\label{eq:umq}
u_\sm=1+\sqrt{(1-u_\sp)^2-2q}\,.
\eeq
Thus, the physical range for range for heat release $q$ and unburned state $u_\sp$ is the set  
\beq\label{eq:uq}
\mathcal{U}:=\left\{(u_\sp,q)\in\RR^2\left| \; u_\sp\geq0,\,0\leq q \leq\frac{1}{2}(1-u_\sp)^2\right.\right\}\,,
\eeq
which corresponds to $u_\sm\in[1,2]$; see Figure \ref{fig:uq}.
The restriction of $u_\sp$ to positive values is explained in Remark \ref{rem:flux}.
Thus, the adjustable parameters are $u_\sp$ (unburned state), $q$ (heat release), $\mathcal{E}_A$ (activation energy---appearing in the definition of $\phi$), $k$ (reaction rate), and $D$ (diffusion). 
\begin{figure}[ht] 
   \centering
   \includegraphics[width=7cm]{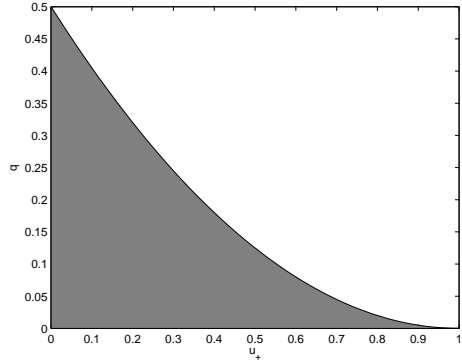} 
   \caption{The physical region $\mathcal{U}$ in $(u_\sp,q)$ space given in equation \eqref{eq:uq} is shaded. Observe that the curved boundary consists of the values of $(u_\sp,q)$ for which the radicand in \eqref{eq:umq} vanishes. Thus, this boundary corresponds to the Chapman--Jouguet state.}
   \label{fig:uq}
\end{figure}

\br
As noted in Figure \ref{fig:uq}, the upper boundary of the physical region $\mathcal{U}$ corresponds to the distinguished CJ state. We denote the corresponding maximum values of $q$ by $q_\mathrm{max}$ (the value evidently depending on $u_\sp$).
These profiles occupy a special place in the theory. For example, we recall that in the inviscid ZND case,  CJ profiles feature slower decay/longer tails.  Specifically, the $z$ coordinate, as always, decays exponentially at the usual rate, but $u$, being related to $z$ by  $u=c + \sqrt{z}$  in this case, decays at half of the rate and therefore generates a tail of twice the length. These long tails can be cause problems for reliable Evans-function computation, and, as described in Appendix \ref{sec:parameter}, some choices of parameters with $q$ near $q_\mathrm{max}$ lead to slowly decaying profiles that are outside of the range of our computation. 
\er
\begin{figure}[p]
\begin{center}
$\begin{array}{cc}
\includegraphics[width=6.5cm]{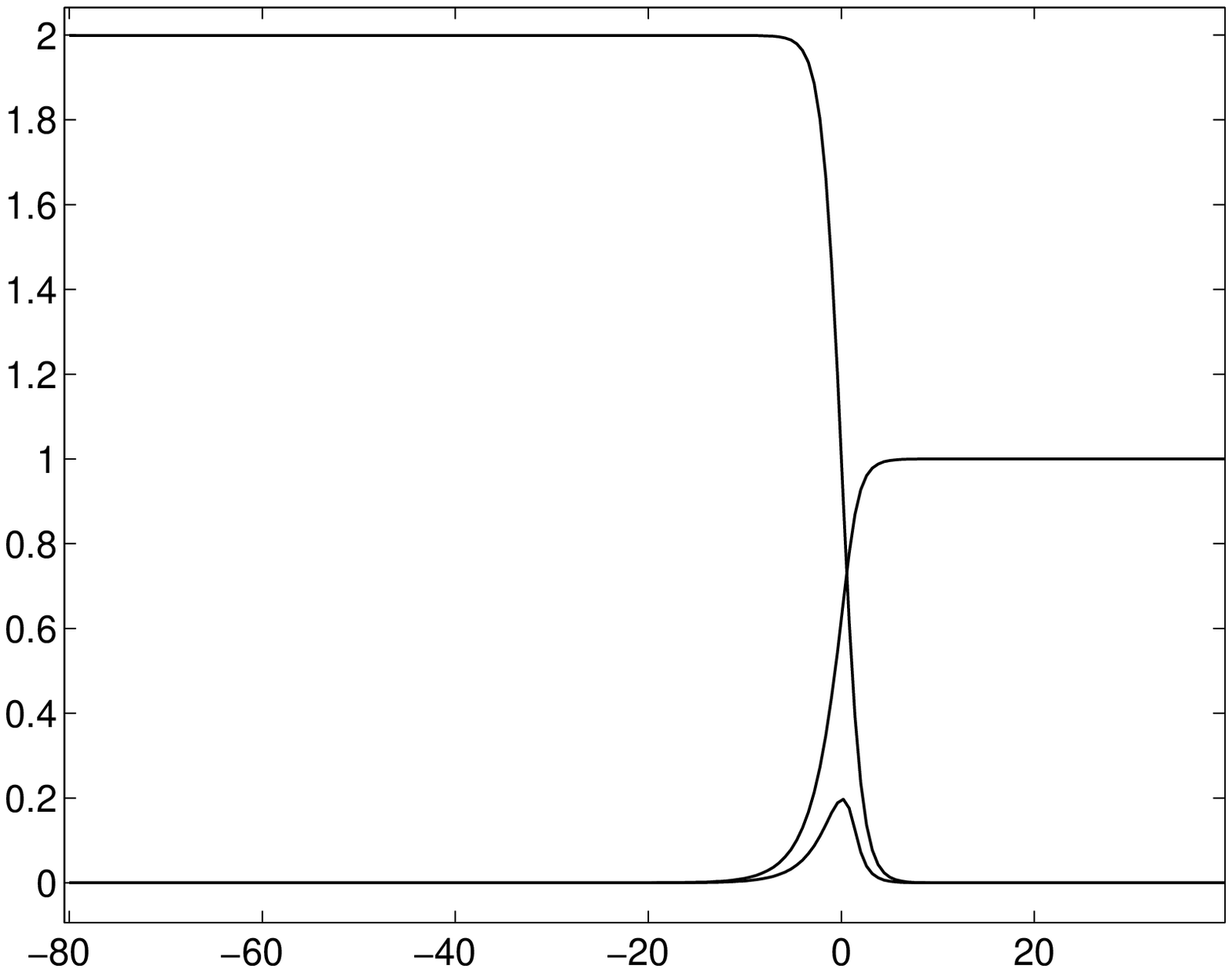} & \includegraphics[width=6.5cm]{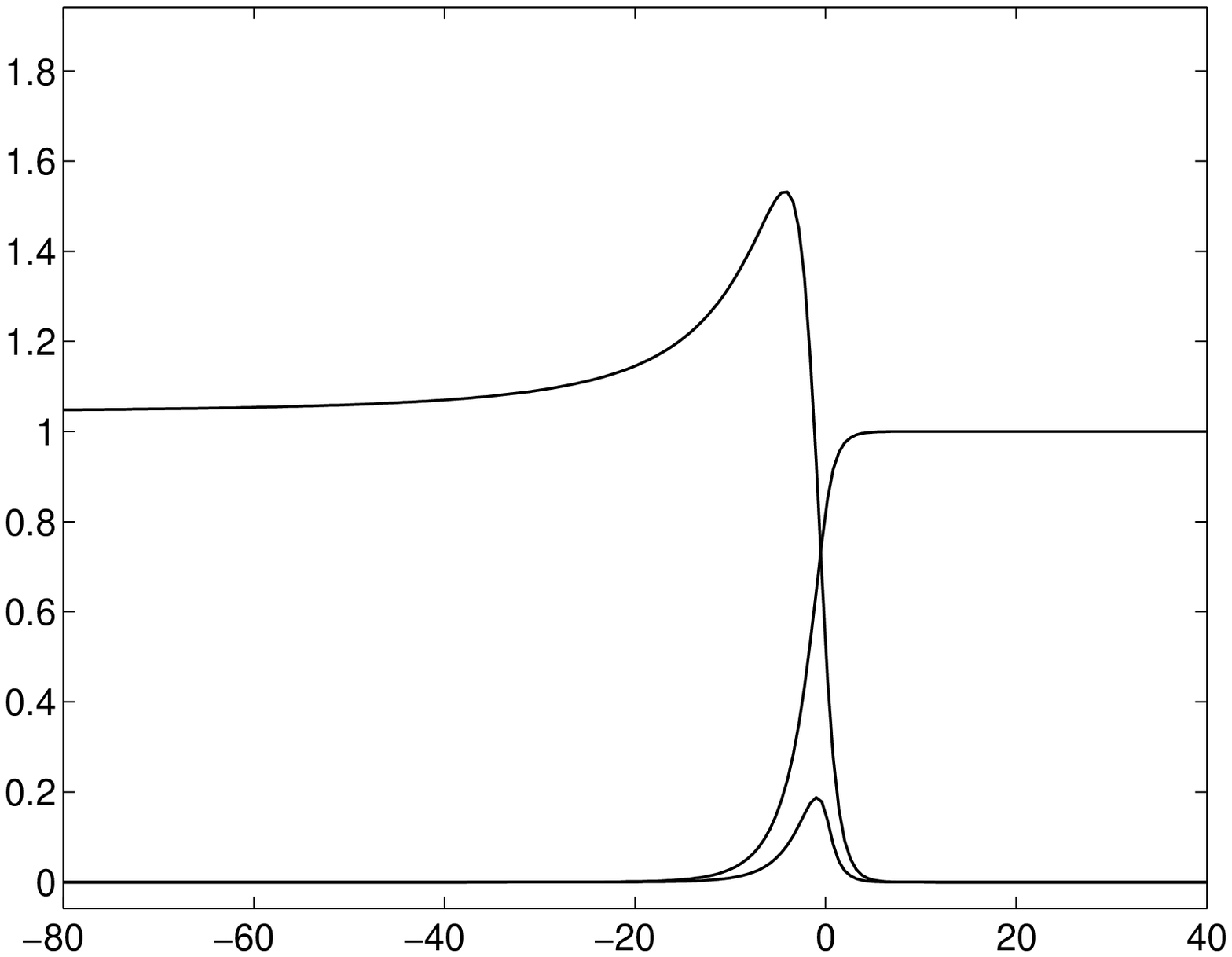} \\
\mbox{\bf (a)} & \mbox{\bf (b)} \\
\includegraphics[width=6.5cm]{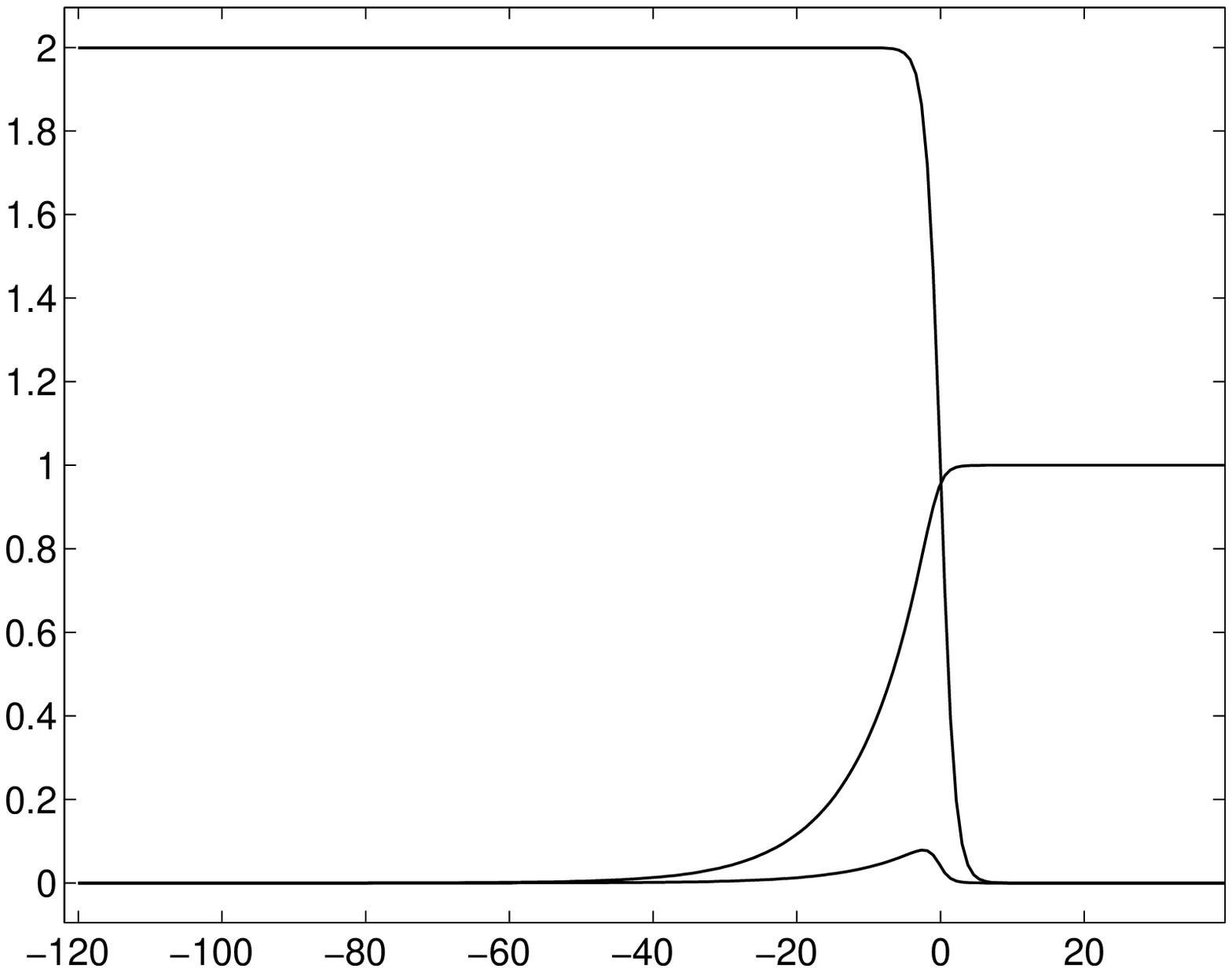} & \includegraphics[width=6.5cm]{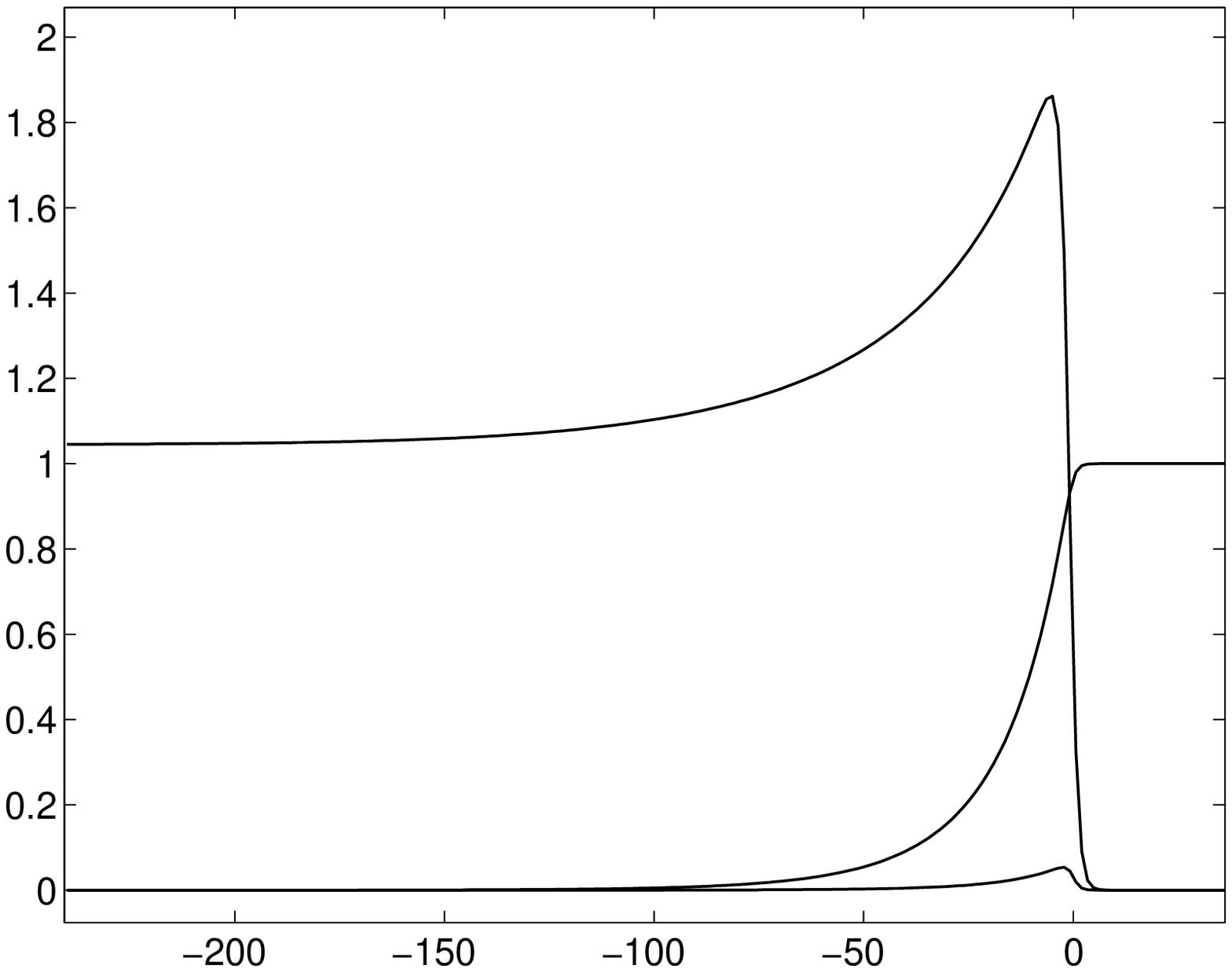} \\
\mbox{\bf (c)} & \mbox{\bf (d)} \\
\includegraphics[width=6.5cm]{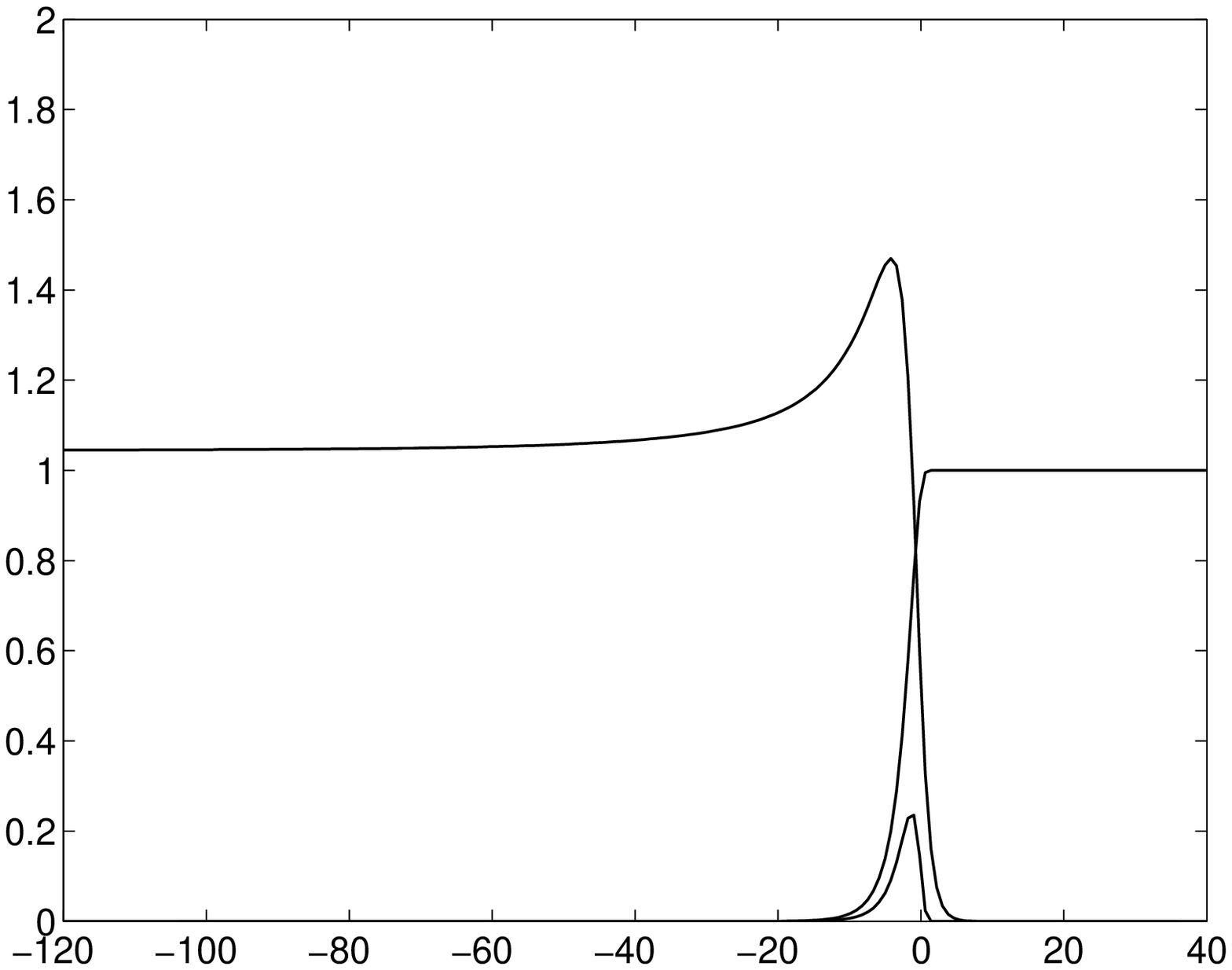} & \includegraphics[width=6.5cm]{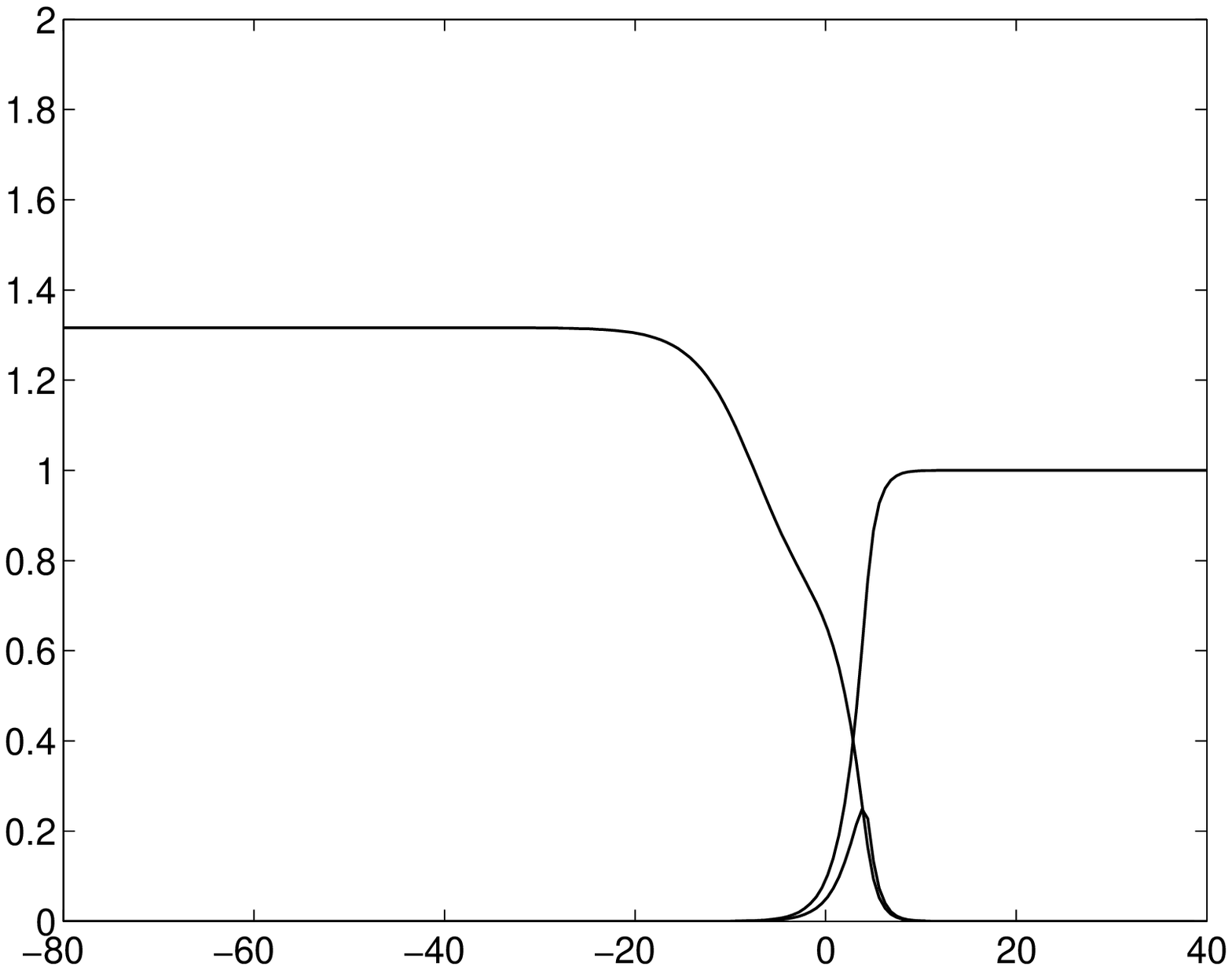} \\
\mbox{\bf (e)} & \mbox{\bf (f)}
\end{array}$
\end{center}
\caption{Graphs of the profile solutions $\hat{u}$, $\hat{y}$, and $\hat{z}$ against $x$ for different parameters.  We consider the intermediate parameter regime ($D=1$, $k=1$, $\mathcal{E}_A=1$), for \textbf{(a)} small $q=0.001$ and \textbf{(b)} large $q=0.499$.  We also consider different limiting cases for large $q$ including \textbf{(c)} large $\mathcal{E}_A$ ($D=1$, $k=1$, $\mathcal{E}_A=4$ and $q=0.499$), \textbf{(d)} small $k$ ($D=1$, $k=0.125$, $\mathcal{E}_A=1$ and $q=0.499$), \textbf{(e)} small $D$ ($D=0.125$, $k=1$, $\mathcal{E}_A=1$ and $q=0.499$), and \textbf{(f)} small $\mathcal{E}_A$ ($D=1$, $k=1$, $\mathcal{E}_A=0.125$ and $q=0.45$).  We distinguish the three curves by noting that $\hat{u}$ converges to $u_\sm>1$ for large negative values of $x$, $\hat{z}$ converges to unity for large positive values of $x$, and $\hat{y}$ converges to zero at both ends.}
\label{fig:profiles}
\end{figure}

\subsubsection{Profile properties}
\label{sssec:profileprops}
We conclude this section by displaying in Figure \ref{fig:profiles} some numerically computed solutions of \eqref{eq:tw} that illustrate the variety of forms taken by strong-detonation-wave solutions of the Majda model.  In Figure \ref{fig:profiles}\textbf{(a)}--\textbf{(b)}, we look at the intermediate parameter regime for small and large values of $q$.  We note that the small values of $q$ reproduce roughly the same profiles regardless of the other parameters.  However, for large values of $q$ the other parameters provide some variation.  In Figure \ref{fig:profiles}\textbf{(c)}--\textbf{(f)}, we look at several limiting cases, including large and small values of $\mathcal{E}_A$, and small values of $k$ and $D$.

\section{Spectral stability}\label{sec:spectral}
\subsection{Linearized equations \& eigenvalue problem}\label{ssec:eval}
Turning now to our stability analysis, we see that the linear equations obtained by linearizing \eqref{eq:majda3} about $(\hat u, \hat z)$ are 
\begin{subequations}\label{eq:mmlin}
\begin{align}
&u_t-q(k\phi'(\hat u)u\hat z+k\phi(\hat u)z)+((\hat u -1) u)_x=u_{xx}\,,\label{eq:mmlin1}\\
&z_t-z_x=-k\phi'(\hat u)u\hat z-k\phi(\hat u)z+Dz_{xx}\,.\label{eq:mmlin2}
\end{align}
\end{subequations}
In \eqref{eq:mmlin}, $u$ and $z$ now denote perturbations. The eigenvalue equations corresponding to this linear system are thus 
\begin{subequations}\label{eq:eval}
\begin{align}
&u''=\lambda u-q(k\phi'(\hat u)u\hat z+k\phi(\hat u)z)+((\hat u-1) u)'\,,\label{eq:eval1} \\
&z''=D^{-1}\big(\lambda z-z'+k\phi'(\hat u)u\hat z+k\phi(\hat u)z\big),\,.\label{eq:eval2}
\end{align}
\end{subequations}
In \eqref{eq:eval} and hereafter $'=\dif/\dif x$.
Alternatively, upon substituting $Dz''-\lambda z+z'=k\phi'(\hat u)u\hat z + k\phi(\hat u)z$ from \eqref{eq:eval2} into \eqref{eq:eval1}, we may rewrite \eqref{eq:eval1} as
\begin{equation}
u''=\lambda(u+qz)-qz'-qDz''+(\alpha u)'\,.\label{eq:alteval1}
\end{equation}
The first step towards constructing the Evans function is to write \eqref{eq:eval} as a first-order system. To do so, we define
$W:=(u,z,u',z')^\tr$, so that the eigenvalue equation becomes
\beq
W'=\mathbb{A}(x;\lambda)W\,,
\eeq
where
\beq
\label{eq:Amatrix}
\mathbb{A}(x;\lambda) =
\begin{bmatrix}
0 & 0 & 1 & 0  \\
0 & 0 & 0 & 1 \\
\lambda + \hat u_x-qk\phi'(\hat u)\hat z & -qk\phi(\hat u) & \hat u-1 & 0 \\
D^{-1}k\phi'(\hat u) \hat z & D^{-1}\lambda + D^{-1}k\phi(\hat u) & 0 & -D^{-1}
\end{bmatrix}\,.
\eeq
For strong detonations, as in the shock case,
it is advantageous to work with the integrated equations.  
This has the effect of removing the translational zero eigenvalue.\footnote{For weak detonations, 
there is no advantage.}
We define $w':=u+qz$ so that \eqref{eq:alteval1} becomes 
\begin{equation}
u''=\lambda w'-qz' -qDz''+((\hat u-1) u)'.\label{eq:3altdn0eval1}
\end{equation}
which can be integrated so that the eigenvalue equation becomes
\begin{subequations}\label{eq:intevalsystem}
\begin{align}
u'&=\lambda w-qz-qDz'+(\hat u-1) u\,, \\
w'&=u+qz\,,\\
z''&=D^{-1}\big(\lambda z-z'+k\phi'(\hat u)u\hat z+k\phi(\hat u)z\big)\,.
\end{align}
\end{subequations}
In matrix form with unknown $X:=(u,w,z,z')^\tr$, \eqref{eq:intevalsystem} takes the form 
\beq
\label{eq:int_eval_ode}
X'=\mathbb{B}(x;\lambda) X
\eeq
with
\begin{equation}\label{eq:bmatrix}
\mathbb{B}(x;\lambda):=\begin{bmatrix}
\hat u-1 & \lambda & -q & -qD \\
1 & 0 & q & 0 \\
0 & 0 &0 & 1\\
D^{-1}k\phi'(\hat u)\hat z & 0 & D^{-1}(\lambda+k\phi(\hat u)) & -D^{-1}
\end{bmatrix}.
\end{equation}
In either case, we have cast the eigenvalue problem as a variable-coefficient linear system of first-order ODEs with a coefficient matrix which depends on the spectral parameter $\lambda$. Notably, the coefficient matrices decay exponentially fast, by virtue of Lemma~\ref{lem:expdecay}, to constant (with respect to $x$) matrices. We denote these two limiting systems by 
\beq\label{eq:limit}
W'=\mathbb{A}_\spm(\lambda)W,\;\;\text{and}\;\;
X'=\mathbb{B}_\spm(\lambda)X\,.
\eeq
It is precisely in this setting that the Evans function can be constructed. Because the construction of the Evans function for the Majda model has been described in detail elsewhere \cite{LRTZ_JDE07}, we shall omit virtually all of the details of the construction and content ourselves with utilizing its fundamental properties. 
For an introduction to the Evans function in the setting
of conservation and balance laws, see 
\cites{GZ_CPAM98,PW_PTRSL92} or the survey articles \cites{Z_IMA,Z_hand,Z_num};
for a general introduction, see,
e.g., \cite{AGJ_JRAM90} or the survey
article \cite{S_HDS02}.

\subsection{High-frequency bounds}\label{ssec:hfb}
We note that the integrated equations \eqref{eq:intevalsystem} can be written as 
\begin{subequations}
\label{eq:eval5}
\begin{align}
&\lambda w = (1-\hat{u}) w' + q\hat{u} z + q(D-1)z' + w''\,, \label{eigenvalue5:a}\\
&\lambda z + k (\phi(\hat{u}) - q \phi'(\hat{u})\hat{z}) z = z' + k \phi'(\hat{u})\hat{z} w' + D z''\,. \label{eigenvalue5:b}
\end{align}
\end{subequations}
Using \eqref{eq:eval5}, we show by an energy estimate that any unstable eigenvalue of the integrated eigenvalue equations must lie in a bounded region of the unstable half plane.  

\begin{prop}[High-frequency bounds]\label{prop:hfb}
Any eigenvalue $\lam$ of \eqref{eq:eval5} with nonnegative real part satisfies
\beq\label{eq:bound}
\re\lambda+|\im\lambda|\leq \max\left\{ 4, \frac{1}{4 D} + \left( \frac{1}{4} + \frac{1}{2}|D-1|^2\right) k L + k M \right\}
\eeq
where 
\beq\label{eq:mdef}
L:= \sup_{x\in\RR}{\phi'(\hat u(x))\hat z} \quad\text{and}\quad M:=\sup_{x\in\RR} \left((1+q) \phi'(\hat{u}) \hat{z}-\phi(\hat{u})\right).
\eeq
\end{prop}

\begin{proof}
We multiply \eqref{eigenvalue5:a} by $\bar{w}$ and \eqref{eigenvalue5:b} by $\bar{z}$ and integrate (we integrate the $w''\bar w$, $z''\bar{z}$ and $z'\bar{w}$ terms by parts) to give
\begin{subequations}
\label{energy}
\begin{align}
&\lambda \intr |w|^2 + \intr |w'|^2 = \intr (1-\hat{u}) w'\bar{w} + q \intr \hat{u} z\bar{w} - q(D-1)\intr z\bar{w}', \label{energy:a}\\
&\lambda \intr |z|^2 + D \intr |z'|^2 + k \intr (\phi(\hat{u})- q \phi'(\hat{u}) \hat{z}) |z|^2   = \intr z'\bar{z} - k \intr\phi'(\hat{u})\hat{z} w'\bar{z}. \label{energy:b}
\end{align}
\end{subequations}
Taking the real part of \eqref{energy}, we find
\begin{subequations}
\label{eq:real}
\begin{align}
&\re\lambda \intr |w|^2 + \intr |w'|^2 =\re \left(\intr (1-\hat{u}) w'\bar{w} + q \intr \hat{u} z\bar{w} - q(D-1)\intr z\bar{w}' \right), \label{real:a}\\
&\re\lambda \intr |z|^2 + D \intr |z'|^2 + k \intr (\phi(\hat{u})- q \phi'(\hat{u}) \hat{z}) |z|^2 = -\re \left(k \intr\phi'(\hat{u})\hat{z} w'\bar{z}\right). \label{real:b}
\end{align}
\end{subequations}
Similarly, taking the imaginary part of \eqref{energy}, we observe
\begin{subequations}
\label{eq:imag}
\begin{align}
&\im\lambda \intr |w|^2 = \im\left( \intr (1-\hat{u}) w'\bar{w} + q \intr\hat{u} z\bar{w} - q(D-1) \intr z\bar{w}'\right)\,, \label{imag:a}\\
&\im\lambda \intr |z|^2 = \im\left(\intr z'\bar{z} - k \intr \phi'(\hat{u})\hat{z} w'\bar{z}\right) = 0\,. \label{imag:b}
\end{align}
\end{subequations}
Combining \eqref{eq:real} and \eqref{eq:imag}, we see that
\begin{equation}
\begin{split}
\label{eq:ri-a}
&\big(\re\lambda + |\im\lambda|\big) \intr |w|^2 + \intr |w'|^2 \leq\sqrt{2} q \intr \hat{u} |z||w|\\
&\qquad + \sqrt{2} q |D-1| \intr |z| |w'| + \sqrt{2} \intr |1-\hat{u}| |w'| |w|,
\end{split}
\end{equation}
and
\begin{equation}
\begin{split}
\label{eq:ri-b}
&\big(\re\lambda + |\im\lambda|\big) \intr |z|^2 + k \intr (\phi(\hat{u})- q \phi'(\hat{u}) \hat{z}) |z|^2 \\
&\qquad + D \intr |z'|^2 \leq \intr |z'| |z| +  \sqrt{2} k \intr |\phi'(\hat{u})\hat{z}| |w'| |z|\,.
\end{split}
\end{equation}
Using Young's inequality (several times) together with the assumption that $\re\lambda\geq 0$, we find that inequalities \eqref{eq:ri-a} and \eqref{eq:ri-b} imply
\begin{equation}
\begin{split}
\label{eq:young-a}
&\big(\re\lambda + |\im\lambda|\big) \intr |w|^2 + \intr |w'|^2 \leq \sqrt{2} q \|\hat{u}\|_\infty \intr \left(\varepsilon_1 |z|^2 + \frac{|w|^2}{4\varepsilon_1}\right) \\
&\qquad + \sqrt{2} q |D-1| \intr \left(\varepsilon_2 |z|^2 + \frac{|w'|^2}{4\varepsilon_2}\right) + \norm{1 - \hat{u}}{\infty} \intr \left(\varepsilon_3 |w'|^2 +\frac{|w|^2}{4\varepsilon_3}\right)
\end{split}
\end{equation}
and
\begin{equation}
\label{eq:young-b}
\begin{split}
&\big(\re\lambda + |\im\lambda|\big) \intr |z|^2 + k \intr (\phi(\hat{u})- q \phi'(\hat{u}) \hat{z}) |z|^2 + D \intr |z'|^2 \\
&\qquad \leq \intr \left(\varepsilon_4 |z'|^2 + \frac{|z|^2}{4\varepsilon_4}\right) + \sqrt{2} L \intr \left(\varepsilon_5 |w'|^2 + \frac{|z|^2}{4\varepsilon_5}\right)\,.
\end{split}
\end{equation}
We multiply \eqref{eq:young-b} by  $\Theta>0$ and add the result to \eqref{eq:young-a}. The result is 
\begin{equation}
\label{eq:bigdaddy}
\begin{split}
&\big(\re\lambda + |\im\lambda|\big) \left( \intr |w|^2 + \Theta |z|^2 \right) + k \intr \Phi(x) |z|^2 + \intr |w'|^2 + \Theta D \intr |z'|^2\\
&\qquad \leq \intr R_1(x)\Theta |z|^2 + \varepsilon_4 \Theta  \intr |z'|^2 + R_2\intr |w'|^2 +  R_3 \intr |w|^2 \,.
\end{split}
\end{equation}
where 
\begin{align*}
\Phi(x)&=(\phi(\hat{u})- q \phi'(\hat{u})\hat{z}) \,,\\
R_1(x)&= \frac{\sqrt{2} \varepsilon_1 q \|\hat{u}\|_\infty}{\Theta}  + \frac{\sqrt{2} \varepsilon_2 q |D-1|}{\Theta} + \frac{1}{4\varepsilon_4} + \frac{\sqrt{2} k \phi'(\hat{u})\hat{z}}{4\varepsilon_5}\,,\\
R_2&=\sqrt{2}\left(\frac{q |D-1|}{4\varepsilon_2} + \varepsilon_3 \|1-\hat{u}\|_\infty + \varepsilon_5 \Theta k L \right)\,,\\
\intertext{and}
R_3&=\sqrt{2}\left(\frac{q \norm{\hat{u}}{\infty}}{4\varepsilon_1} + \frac{\norm{1-\hat{u}}{\infty}}{4\varepsilon_3}\right)\,.
\end{align*}
%
%
Finally, to simplify \eqref{eq:bigdaddy}, we choose
\begin{align*}
\varepsilon_1 &= \frac{\sqrt{2}}{8} &
\varepsilon_2 &= \sqrt{2} q |D-1| \\
\varepsilon_3 &= \frac{\sqrt{2}}{8 \|\hat{u}-1\|_\infty} &
\varepsilon_4 &= D \\
\varepsilon_5 &= \frac{\sqrt{2}}{4} &
\Theta &= (kL)^{-1}\,,
\end{align*}
where $L$ and $M$ are as in \eqref{eq:mdef}.   We also note that $\|\hat{u}\|_\infty \leq 2$, $\|1-\hat{u}\|_\infty\leq 1$, and $q\leq 1/2$.  Thus, we have
\[
  \big(\re\lambda + |\im\lambda|\big)\intr (|w|^2 + \Theta |z|^2) \leq 4 \intr |w|^2 +C \intr \Theta |z|^2\,,
\]
where
\[
C:= \left( \frac{1}{4 D} + \left( \frac{1}{4} + \frac{1}{2}|D-1|^2\right) k L + k M \right)\,.
\]
The result follows.
\end{proof}

\begin{remark}
We easily obtain the following crude bounds on $L$ and $M$:
\begin{align*}
L & \leq \sup_{x\in\RR}{\phi'(\hat u(x))} \leq \phi'\left(u_\mathrm{ig}+\frac{\mathcal{E}_A}{2}\right) = \frac{4}{\mathcal{E}_A}\phi\left(u_\mathrm{ig}+\frac{\mathcal{E}_A}{2}\right) \leq \frac{4}{\mathcal{E}_A} \me^{-2} \approx \frac{0.5413}{\mathcal{E}_A}\,,\\
M&\leq\sup_{x\in\RR}{(1+q) \phi'(\hat{u})} \leq \frac{6}{\mathcal{E}_A} \me^{-2} \approx \frac{0.8120}{\mathcal{E}_A}\,.
\end{align*}
Numerically, we find for typical parameters that both $L$ and $M$ are less than $1/4$, and thus for moderate values of $\mathcal{E}_A$, $k$ and $D$, we have that the high-frequency bounds satisfy $\re\lambda + |\im\lambda|\leq 4$. 
\end{remark}

\subsection{Evans Function}\label{ssec:evans}
The Evans function $E(\lambda)$, defined as a Wronskian of decaying solutions at $x=\pm\infty$ of the eigenvalue equation \eqref{eq:intevalsystem}, is an analytic function of the spectral parameter $\lambda$ for $\lambda$ in the right half plane. The fundamental property of $E$, in addition to its analyticity, that we exploit here is that 
\[
E(\lambda_0)=0\Leftrightarrow\;\text{$\lambda_0$ is an eigenvalue of the linearized operator $L$.}
\]
While the Evans function is generally 
complicated
to compute analytically, it can readily be computed numerically \cite{HZ_PD06}. Since the Evans function is analytic in the region of interest, we shall numerically compute its winding number in the right-half plane. This process will allow us to systematically search for zeros of $E$ (and hence unstable eigenvalues, recall Proposition \ref{prop:lrtz} above) in the unstable half plane. The origin of this approach to spectral stability can be found in the work of Evans and Feroe \cite{EF_MB77}. These ideas have been applied to a variety of systems since; see, e.g., \cites{PSW_PD93,AS_NW95,B_MC01,BDG_PD02,HLZ_ARMA09}.

Techniques for the numerical approximation of the Evans function have been described in detail elsewhere \cites{B_MC01,HSZ_NM06,HZ_PD06,STABLAB}, so we only outline the important aspects of the computation here. 
\begin{description}
\item[Step 1. Profile] The traveling-wave equation \eqref{eq:tw} is a nonlinear two-point boundary-value problem posed on the whole line. To compute an approximation of the profile, it is necessary to truncate the problem to a finite computational domain $[-X_\sm,X_\sp]$. We use \textsc{MATLAB}'s boundary-value solver, an adaptive Lobatto quadrature scheme \cite{bvp6c}, and we supply appropriate projective boundary conditions at $X_\spm$. The  values for plus and minus spatial computational infinity, $X_\spm$, must be chosen with some care. Writing the traveling-wave equation \eqref{eq:tw} as $\hat U'=F(\hat U)$ together with the condition that $\hat U\to U_\spm$ as $\xi\to\pm\infty$, the typical requirement is that $X_\spm$ should be chosen so that $ |\hat U(\pm X_\spm)-U_\spm|$ is within a prescribed tolerance of $10^{-3}$.  We also set the errors on the solver to be {\tt RelTol=1e-8} and {\tt AbsTol=1e-9}.

We remark also that most of the profile solutions were found by continuation as it would have been difficult otherwise to provide an easy starting guesses to the boundary-value solver; see Figure \ref{fig:profiles}. Thus, an important aspect of the computational Evans-function approach we use here is the ability to continue the profile solutions throughout parameter space.
\item[Step 2. High-frequency bounds] Upon the completion of the first step, the next task is to compute the high-frequency spectral bounds supplied by Proposition \ref{prop:hfb}. This amounts to the evaluation of the quantities $M$ and $L$ in \eqref{eq:mdef}. With these quantities in hand, we may choose a positive real number $R$ sufficiently large that there are no eigenvalues of \eqref{eq:intevalsystem} outside of the domain 
\[
B^\sp_R:=B(0,R)\cap\{\re\lam\geq 0\}\,.
\]
We have thus reduced the problem of verifying the Evans condition \eqref{eq:evans_condition} to the problem of showing that the Evans function does not vanish\footnote{The use of integrated coordinates has removed the zero at the origin.} in the bounded region $B_R^\sp$. 
\item[Step 3. Evans function]
The evaluation of the Evans function is accomplished by means of the \textsc{STABLAB} package, a \textsc{MATLAB}-based package developed for this purpose \cite{STABLAB}. This package allows the user to choose to approximate the Evans function either via exterior products, as in \cites{AB_NM02,B_MC01,AS_NW95}, or by a polar-coordinate (``analytic orthogonalization'') method \cite{HZ_PD06}, which is used here. Importantly, since our search for zeros is based on the analyticity of the Evans function, Kato's method \cite{Kato}*{p. 99} is used to analytically determine the relevant initializing eigenvectors; see \cites{BZ_MC02,BDG_PD02,HSZ_NM06} for details.  Throughout our study, we set the errors on \textsc{MATLAB}'s stiff ODE solver {\tt ode15s} to be {\tt RelTol=1e-6} and {\tt AbsTol=1e-8}.
\item[Step 4. Winding]
Finally, we compute the number of zeros of the Evans function $E$ inside the semicircle $S=\partial B^\sp_R$ by computing the winding number of the image of the curve $S$, traversed counterclockwise, under the analytic map $E$. To do this, we simply choose a collection of test $\lam$-values on the curve $S$, and we sum the changes in $\arg E(\lam)$ as we travel around the semicircle. These changes can be computed directly via the simple relation 
\[
\im\log E(\lam)=\arg E(\lam) \mod 2\pi\,.
\]
We test a posteriori that the change in the argument of $E$ is less than $0.2$ in each step, and we add test values if necessary to achieve this. Most curves resolved well within that tolerance using 120 mesh points in the first quadrant and reflecting by conjugate symmetry for the fourth quadrant.  We recall that by Rouch\'e's theorem, an accurate computation of the winding number is guaranteed as long as the argument varies by less than $\pi/2$ between two test values \cite{Henrici}.
\end{description}

\section{Experiments}\label{sec:experiments}

We now describe our experiments. We recall, from \S\ref{ssec:end} above, that $s=1$, and $u_\sm\in[1,2]$ is given explicitly in \eqref{eq:umq} as a function of $(u_\sp,q)\in\mathcal{U}$.  Although we examined the full range of $u_\sp$ and several values of $u_\mathrm{ig}$, we found that the output was not qualitatively different than setting $u_\sp=0$ and $u_\mathrm{ig}=0.1$, and so we use those values throughout this section. The complete list of parameter values tested is given in Appendix \ref{sec:parameter}.

\subsection{Activation energy}
Bourlioux \& Majda \cite{BM_PT95} have noted that, in the context of the the Euler system for reacting gas, a broad range of phenomena can be observed simply by varying the activation energy $\mathcal{E}_A$ and the heat release $q$. 

\subsubsection{Large activation energy}\label{ssec:bige}
The limit of infinite activation energy is studied in the literature (e.g., Buckmaster \& Neves \cite{BN_PF88}) at least in part because of the simplification it affords. In particular, for a one-step reaction with Arrhenius kinetics, the steady structure of physical ZND waves can be precisely described in this limit; this facilitates the analysis. The hope is that such analysis might plausibly be extrapolated to shed insight into the behavior of waves with large (but finite) activation energies. 
Following common practice, e.g., as in \cite{BZ_majda-znd}, we scale the reaction rate $k$ in part to keep the tail in the computational domain. Based on comparisons to the physical equations, this is the regime in which one would expect to find unstable eigenvalues (if any), and we regard this as one of the principal computations of the paper. We adopt the rescaling for $k$ in terms of activation energy used by Barker \& Zumbrun \cite{BZ_majda-znd} in the for the inviscid Majda-ZND model. That is, we do not compute exactly the half-reaction width as is common in the (inviscid) detonation literature. Rather, we use Barker \& Zumbrun's rough, effective scaling by a factor of $\exp(\mathcal{E}_A/2)$; this simple-to-implement low-cost scaling keeps the reaction width constant within a factor of 1.5 or so, and therefore accomplishes the principal goal of staying within the correct computational regime. Indeed, it is evident that an increase in the value of the activation energy  decreases the size of $\phi$
on the computational domain, so the change in activation energy effectively decreases the value of $k$.  More precisely, it decreases the important 
quantity $k\phi$ that determines the rate of decay.  This explains the need to rescale $k$ for large $\mathcal{E}_A$ in order to keep reaction length approximately fixed.

Representative output is shown in Figure \ref{fig:largeE}.
\begin{figure}[t]
\begin{center}
$\begin{array}{cc}
\includegraphics[width=7cm]{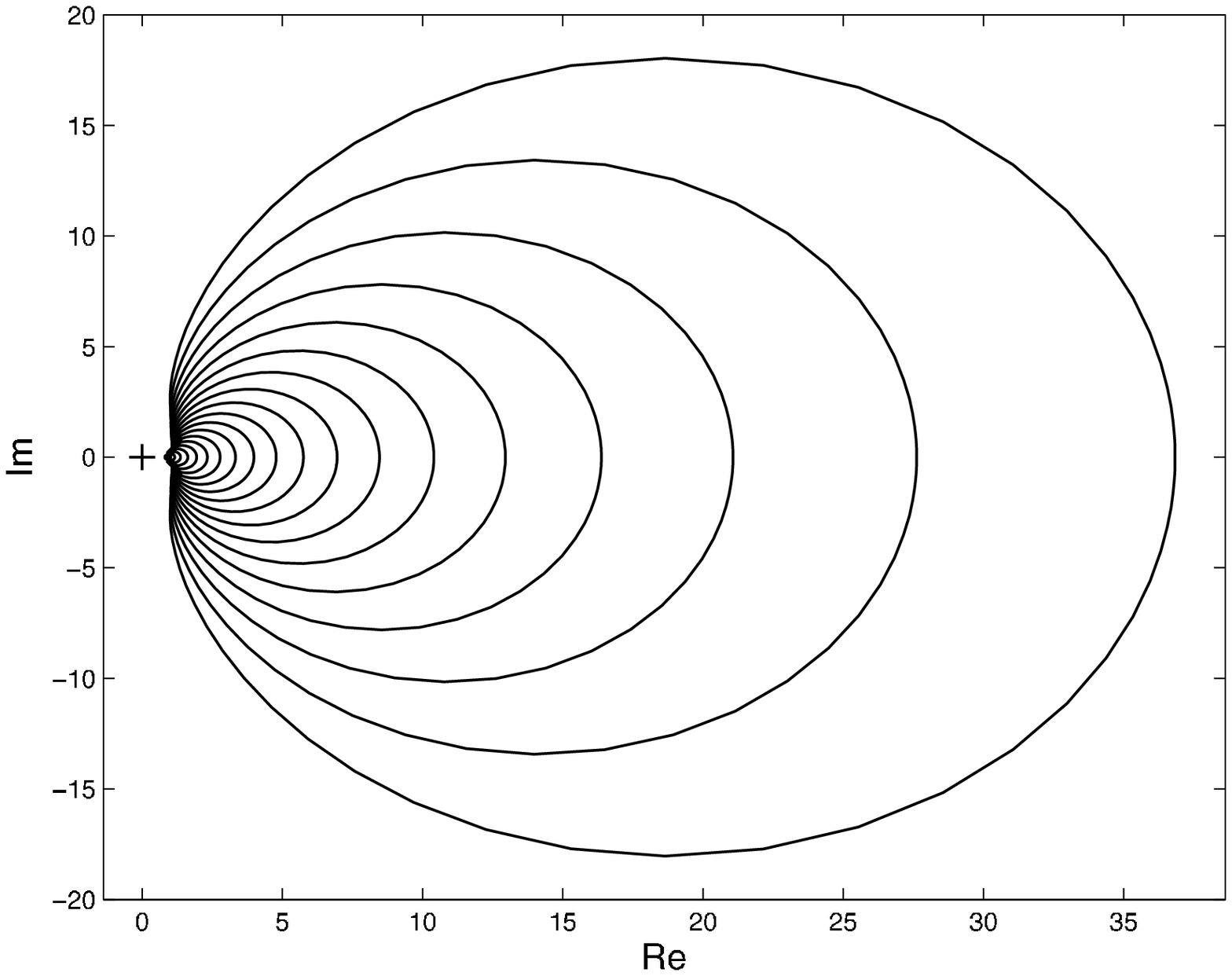} & \includegraphics[width=7cm]{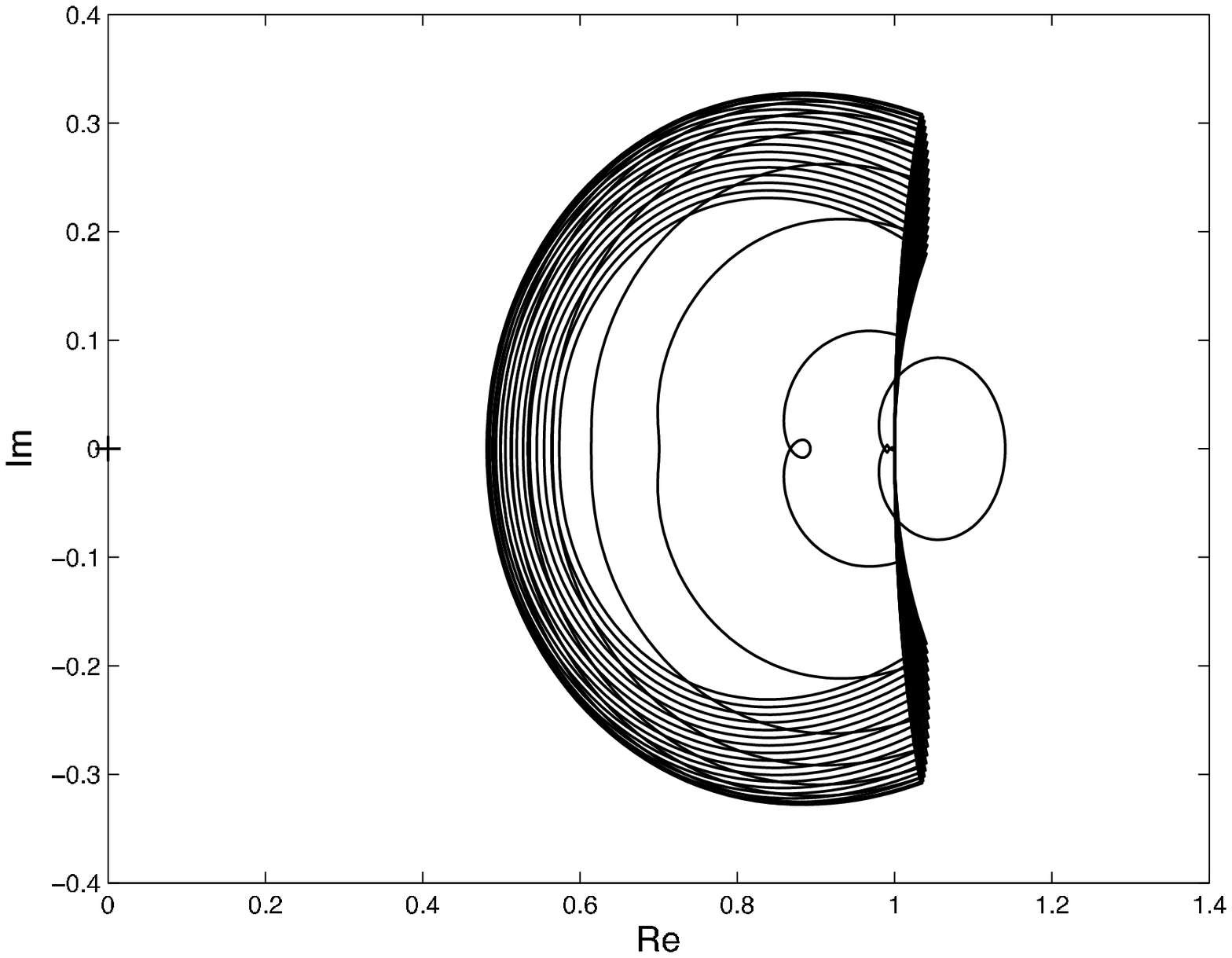} \\
\mbox{\bf (a)} & \mbox{\bf (b)}
\end{array}$
\end{center}
\caption{\textbf{(a)}: Activation energy $\mathcal{E}_A$ ranging from  $1$ to $44$ with $q=0.2$ and $k$ scaled to yield approximately constant reaction length. \textbf{(b)}: Activation energy $\mathcal{E}_A$ ranging from  $1$ to $40$ and $q=0.001$.  The radius of the semicircle contour used in both images is $R=40,\!000$.}
\label{fig:largeE}
\end{figure}
Unfortunately, as $k$ grows, the high-frequency bounds of Proposition \ref{prop:hfb} degrade. For example, for the upper value of $\mathcal{E}_A$ of $44$ shown in Figure \ref{fig:largeE}, the needed radius is of the order of $10^{10}$ and is not feasible, and we lose our ability to rule out the possibility of large eigenvalues outside of our semicircle. Nonetheless, our experiments for a range of contours show very nice, regular behavior with no hints of instability. 

\br
We note that one possible way to get a better handle on this issue would be to use curve-fitting to test for convergence of the Evans function $E$ to its limiting large-$\lam$ asymptotic shape at smaller
radii. This would have the practical effect of revealing some nontrivial cancellation that is not captured by our energy estimate. See, e.g., \cite{BZ_majda-znd}. However, we do not pursue this idea any further here. 
\er
\subsubsection{Intermediate activation energy}

In Figure \ref{fig:midE}, we see a sample of Evans function output for various values of $\mathcal{E}_A$ as $q$ is varied.  For small values of $q$, the rightmost portion of the curves corresponds to the values of the Evans function along the imaginary axis.  In other words, the value of the Evans function is larger along the imaginary axis than around the half-circle of radius $4$. Upon zooming in on the origin, see Figure \ref{fig:midE}\textbf{(b)}, we observe that the output contours feature small ``noses'' pointing toward the origin. These correspond to the output of the imaginary axis for values of heat release $q$ near its maximum possible value of $1/2$, i.e., near the the limiting CJ value.

\begin{figure}[ht]
\begin{center}
$\begin{array}{cc}
\includegraphics[width=7cm]{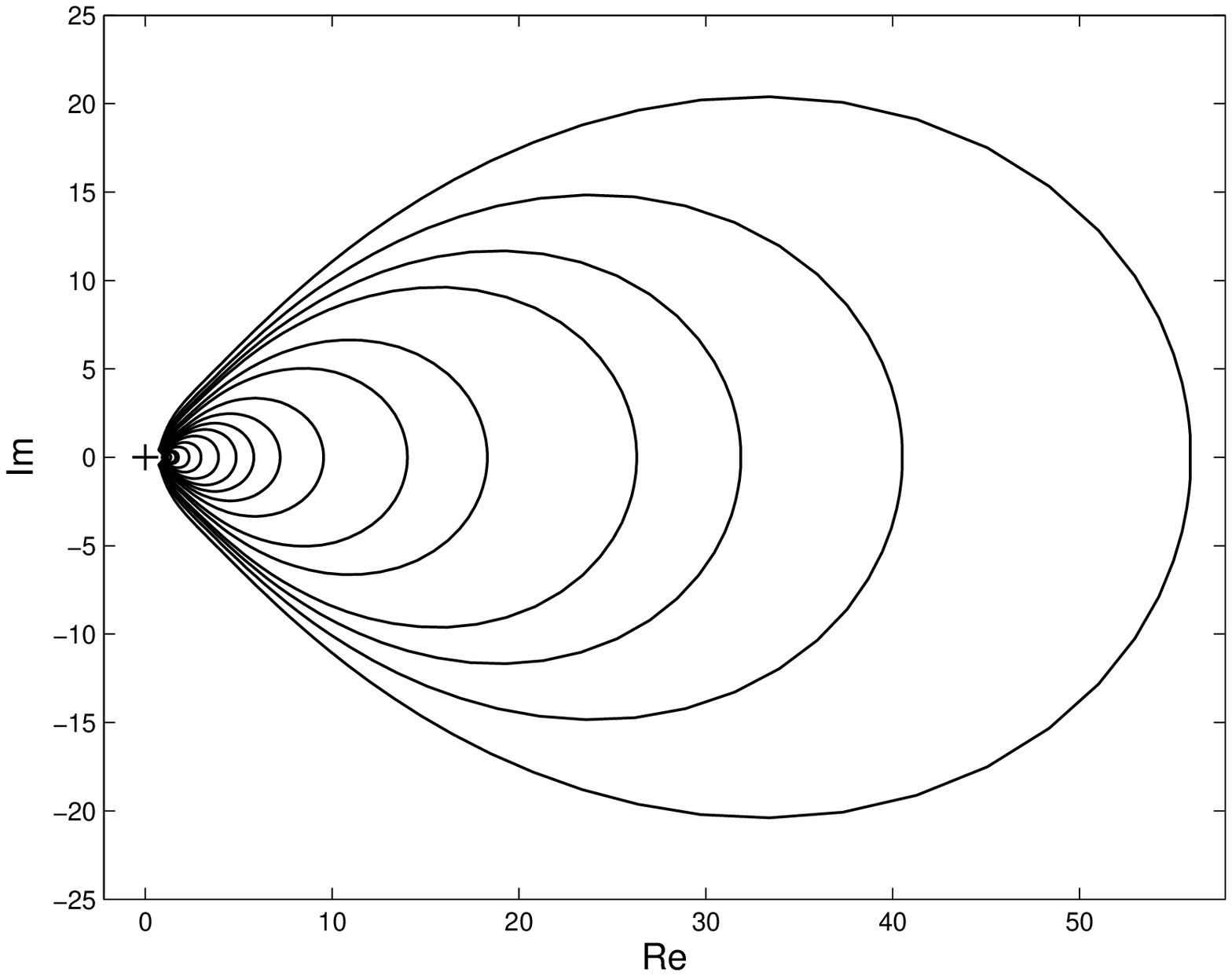} & \includegraphics[width=7cm]{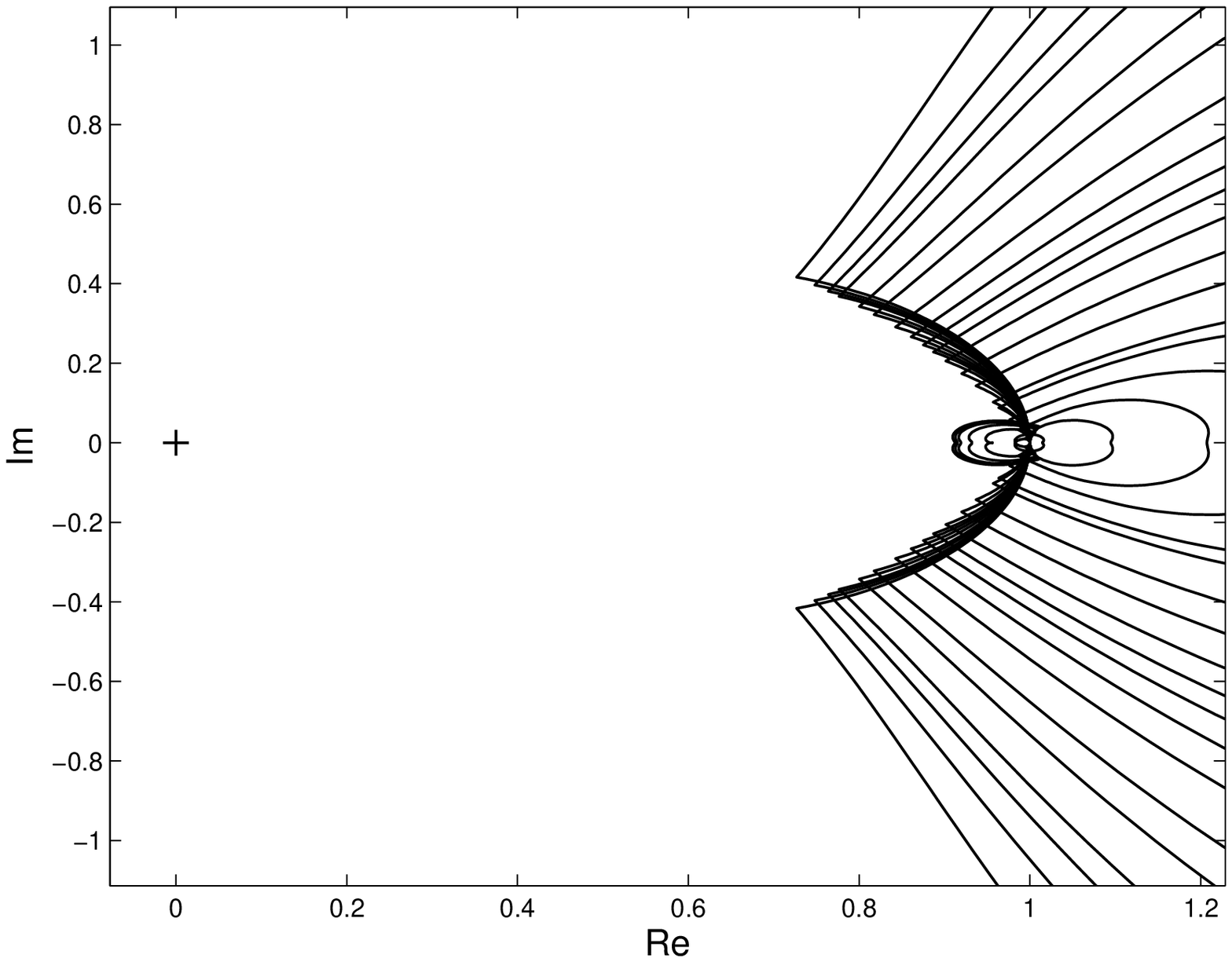} \\
\mbox{\bf (a)} & \mbox{\bf (b)} \\
\includegraphics[width=7cm]{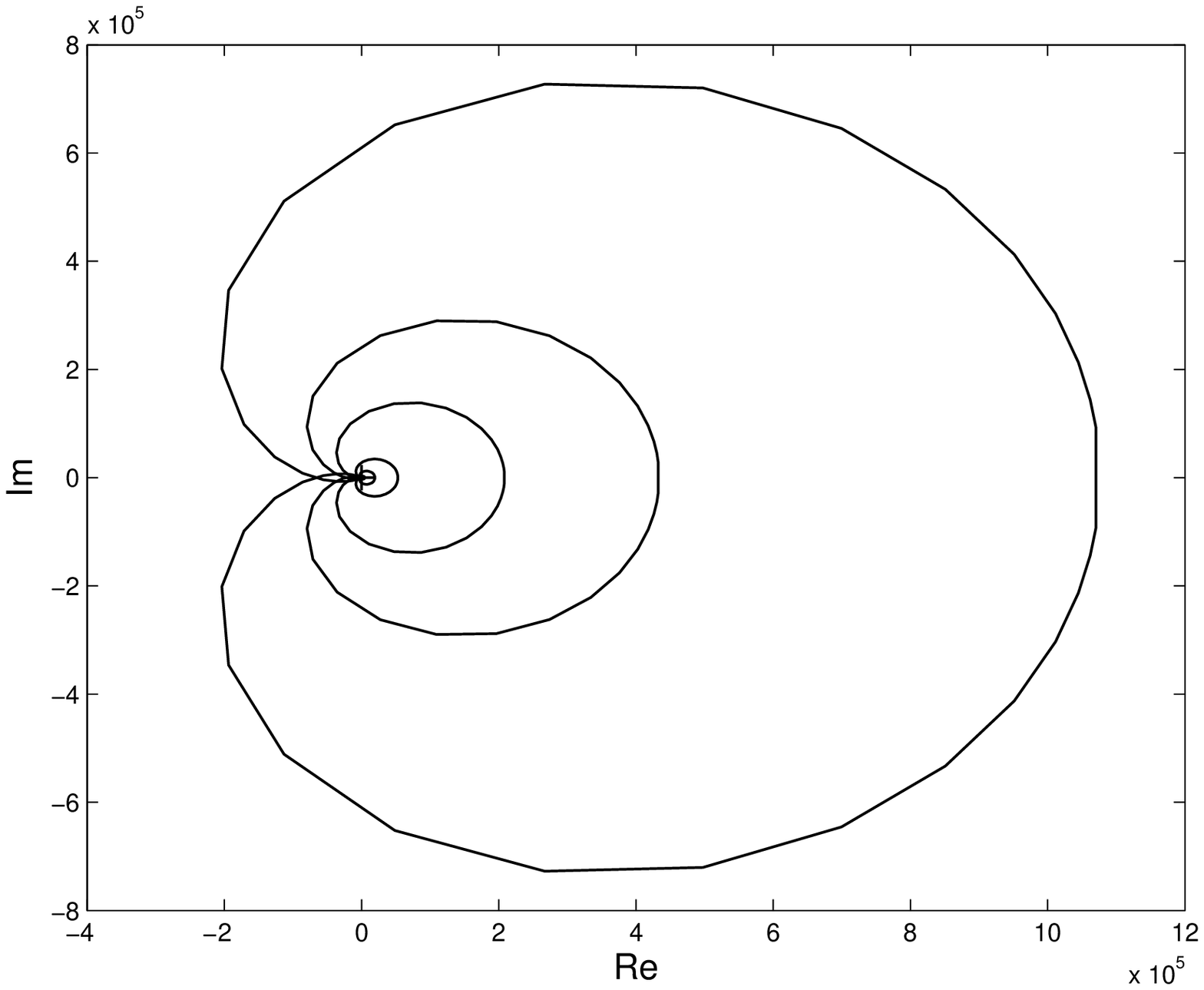} & \includegraphics[width=7cm]{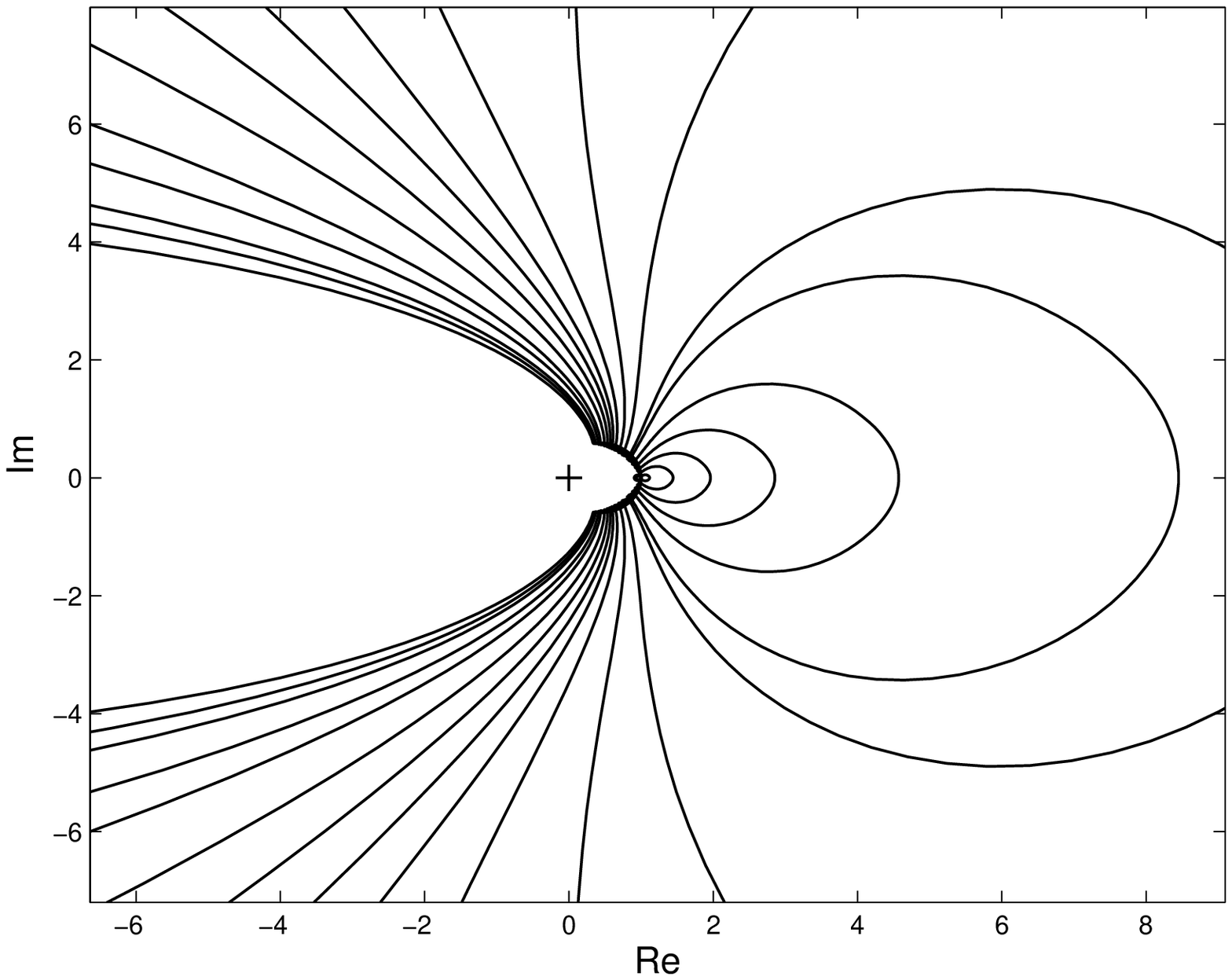} \\
\mbox{\bf (c)} & \mbox{\bf (d)}
\end{array}$
\end{center}
\caption{Evans function output for various values of $\mathcal{E}_A$, the full range of values for $q$ and intermediate values of the other parameters ($D=k=1$).  The input contours contain the high-frequency bounds given in \eqref{eq:bound}, and consist of a half-circle contour in the right half plane centered at the origin with radius $4$. The values of activation energy are \textbf{(a)} $\mathcal{E}_A=2$ and \textbf{(c)}  $\mathcal{E}_A=4$.  Figures \textbf{(b)} and \textbf{(d)} are magnified versions of \textbf{(a)} and \textbf{(c)}, respectively, with particular attention paid to the origin, where we can see that the winding number is zero in \textbf{(b)}.  It is also zero in \textbf{(d)} but the structure is richer.}
\label{fig:midE}
\end{figure}

\subsubsection{Small and zero activation energy}\label{ssec:smalle}
Suggestively, our experiments show that the noses seen in Figure \ref{fig:midE}\textbf{(b)} grow considerably for small values of activation energy.  For example, in Figure \ref{fig:smallE}, we see that for $\mathcal{E}_A=1/4$, these loops have grown substantially, and the continued growth and proximity to the origin is dramatic when $\mathcal{E}_A=1/8$. This behavior hints at a possible instability for sufficiently small values of $\mathcal{E}_A$ and $q\approx 1/2$. Indeed, in Figure \ref{fig:smallE}\textbf{(b)}, we see that the contour almost intersects the origin for $q=0.45$. This raises the question of whether it would indeed intersect the origin at $q$ approaches $1/2$.  Such an event would indicate an unstable eigenvalue crossing into the right half plane.  As it turns out, however, we are unable to get profiles for values of $q$ larger than $q=0.45$. Indeed, as we shall now describe,
this seems to be the boundary of existence of profiles, and not just a numerical difficulty.

We found that a tell-tale sign of the failure of the profile solver is the formation of a ``bench''---a nearly flat portion
of the profile---in $\hat u$; see Figure \ref{fig:profileA} in Appendix \ref{sec:zeroe} for a picture of the bench in the case that $\mathcal{E}_A$ is zero. We were able to get the boundary-value solver to return similar-looking structures for nonzero $\mathcal{E}_A$. However, these ``solutions'' had unacceptably large residuals, and are not to be trusted. An important feature of the benches is that they occur at heights corresponding to $u_\sm^\mathrm{w}$, and the failure of the boundary-value solver as the ``noses'' approached the origin in the complex plane is associated with the approach of $q$ to one of those distinguished values for which a weak detonation profile exists.  

Indeed, to further investigate this phenomenon, which we found for small values of the activation energy $\mathcal{E}_A$, we considered the limiting case of zero activation energy. We give a detailed discussion of this mathematically interesting case in Appendix \ref{sec:zeroe}. Briefly, in the zero-$\mathcal{E}_A$ setting, it is simple to see by ``shooting'' that for $q$ small, the profiles match up nicely to those with small activation energy but that as $q$ increases, the profile start to develop a bench (corresponding to the distinguished $q$ value for which a weak detonation exists). For larger $q$ values no profile of either type exists.  Finally, we recall that the limit of vanishing activation energy 
is  
(essentially) a regular perturbation problem for which we may readily
establish convergence as $\mathcal{E}_A\to 0$ to the $\mathcal{E}_A=0$
flow,
so that the observed behavior with $\mathcal{E}_A=0$ does indeed describe the behavior for $\mathcal{E}_A$ sufficiently small.

\begin{figure}[ht]
\begin{center}
$\begin{array}{cc}
\includegraphics[width=7cm]{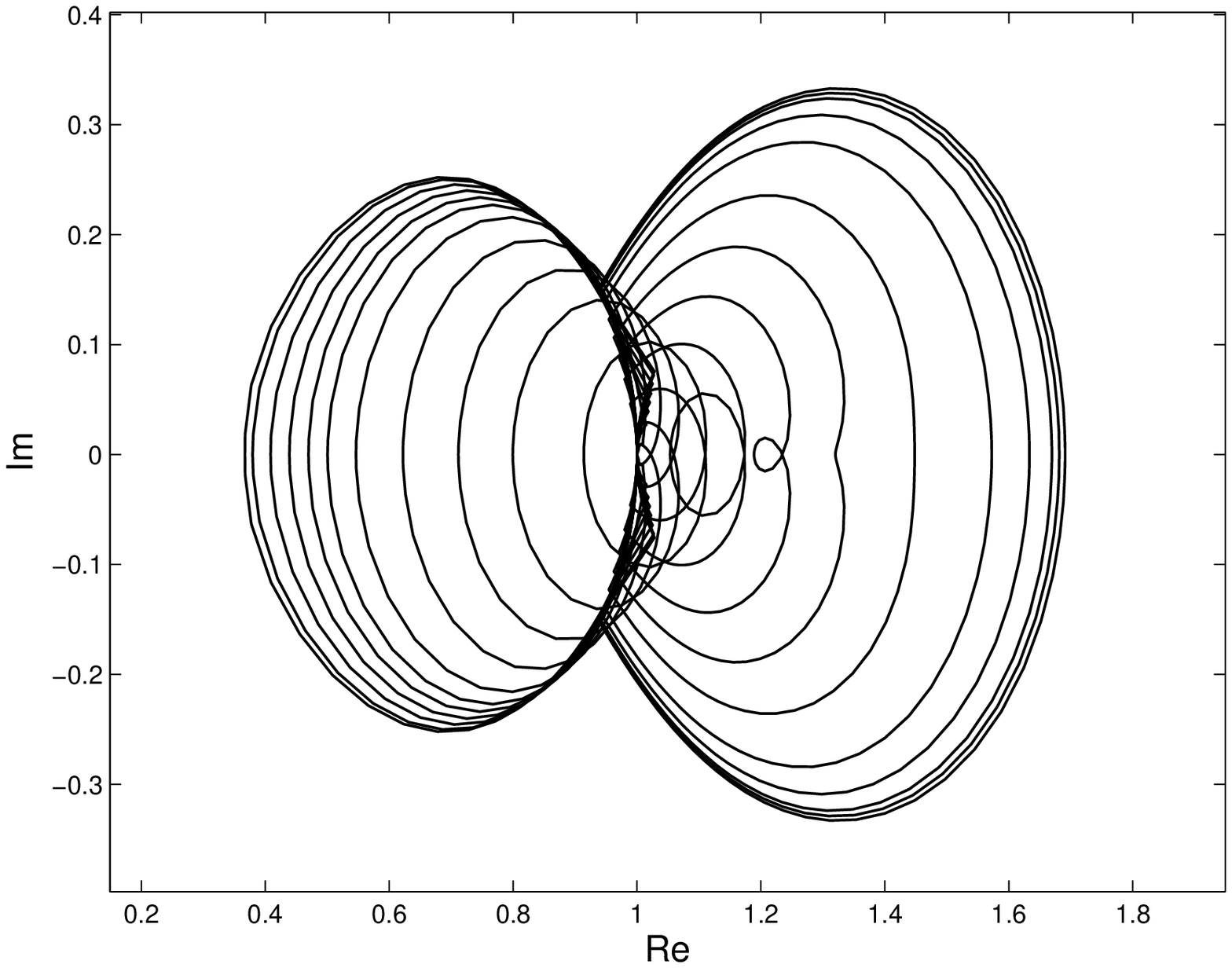} & \includegraphics[width=7cm]{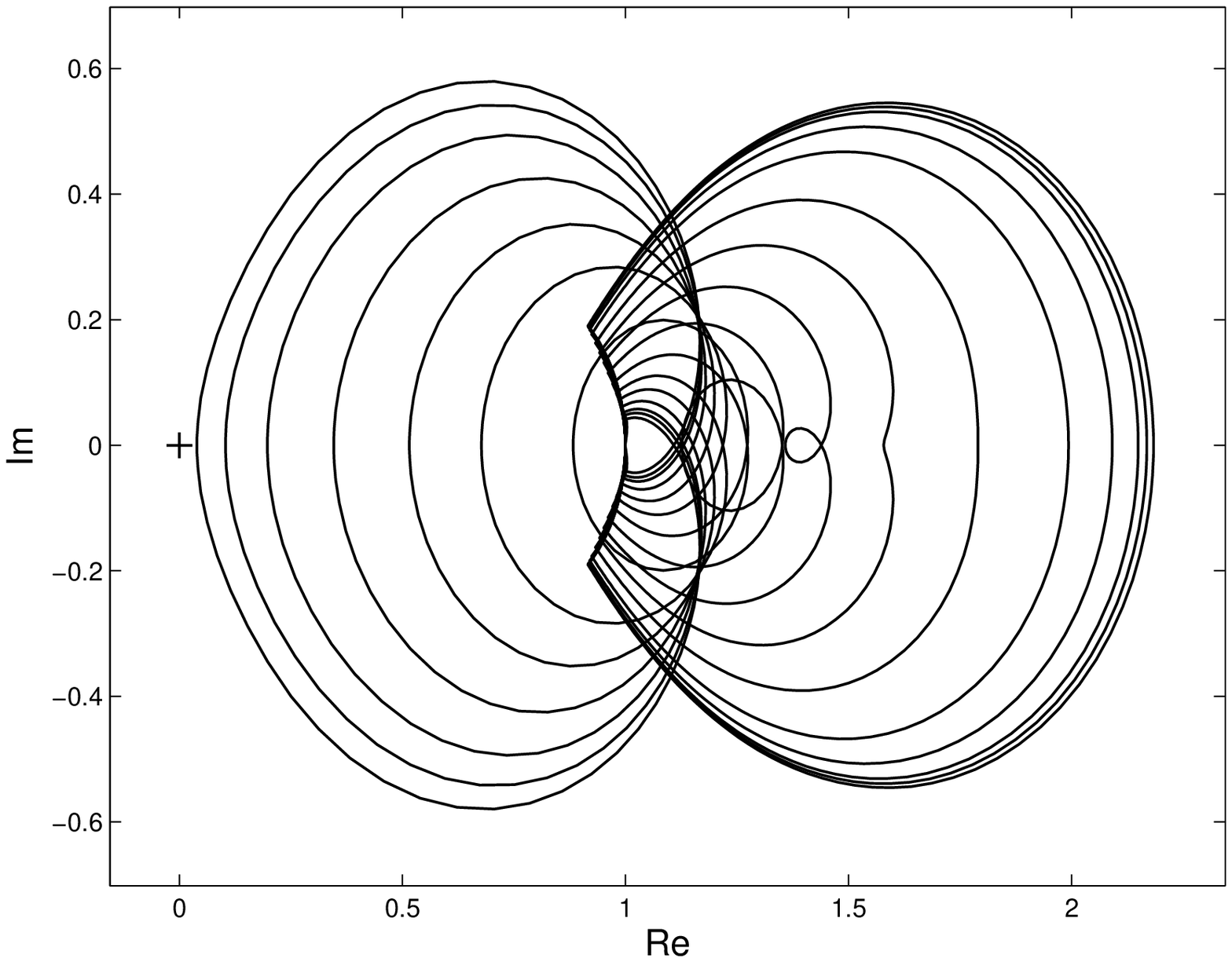} \\
\mbox{\bf (a)} & \mbox{\bf (b)}
\end{array}$
\end{center}
\caption{Examination of the Evans function output as the activation energy is small and the other parameters, $D$ and $k$, are unity.  In particular, we examine the system for \textbf{(a)} $\mathcal{E}_A=1/4$ and $q$ ranging from $q=0.001$ to $q=0.499$ and \textbf{(b)} is for $\mathcal{E}_A=1/8$ and $q$ ranging from $q=0.001$ to $q=0.45$.  The contour radii are \textbf{(a)} a radius of $4$ and \textbf{(b)} a radius of $6$.}
\label{fig:smallE}
\end{figure}


\subsection{Limiting Behavior}\label{ssec:limit}
We observe that both of the limits $D\to 0$ and $k\to 0$ considered in this section are singular limits. As such, it is more appropriate to study them by 
means of an Evans-function analysis involving the formal limiting problem as in \cite{Z_ARMA11}. 

\subsubsection{The ZND Limit ($k\to 0$)}
Also, $\mathcal{E}_A=D=1$, and $k\to 0$ is interesting (ZND limit), though badly conditioned.  In the case of small $k$, we find that the tail on the left hand side gets long.  We also see that the contours do not get closer to the origin in the limit.  in Figure \ref{fig:smallK2}, we see the same qualitative behavior as seen in Figure \ref{fig:midE}\textbf{(d)}.  Specifically, the Evans function output wraps about the origin and then unwraps, thus producing a winding number of zero, but with a richer structure.

\br
Zumbrun's analysis \cite{Z_ARMA11} of the ZND limit shows that zeros of the Evans function $E$ converge in the ZND limit to zeros of the associated Evans function for the Majda-ZND system, call it $E_0$. However, the same analysis shows that the convergence of the functions themselves, namely the convergence of $E$ to $E_0$, is a much more delicate issue. In particular, it can be seen only after a special rescaling which is badly behaved in the limit. 
\er

\begin{figure}[t]
\begin{center}
$\begin{array}{cc}
\includegraphics[width=7cm]{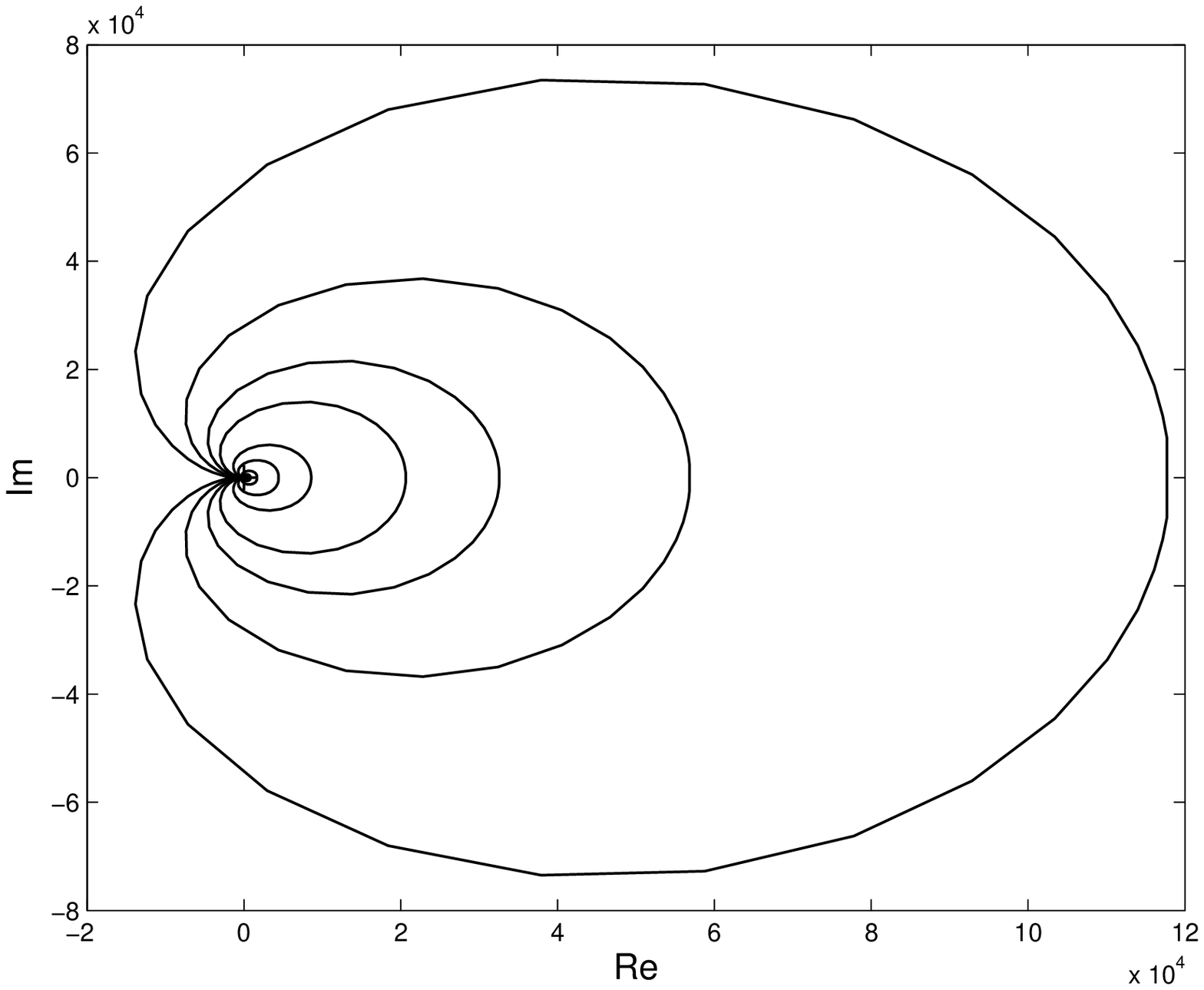} & \includegraphics[width=7cm]{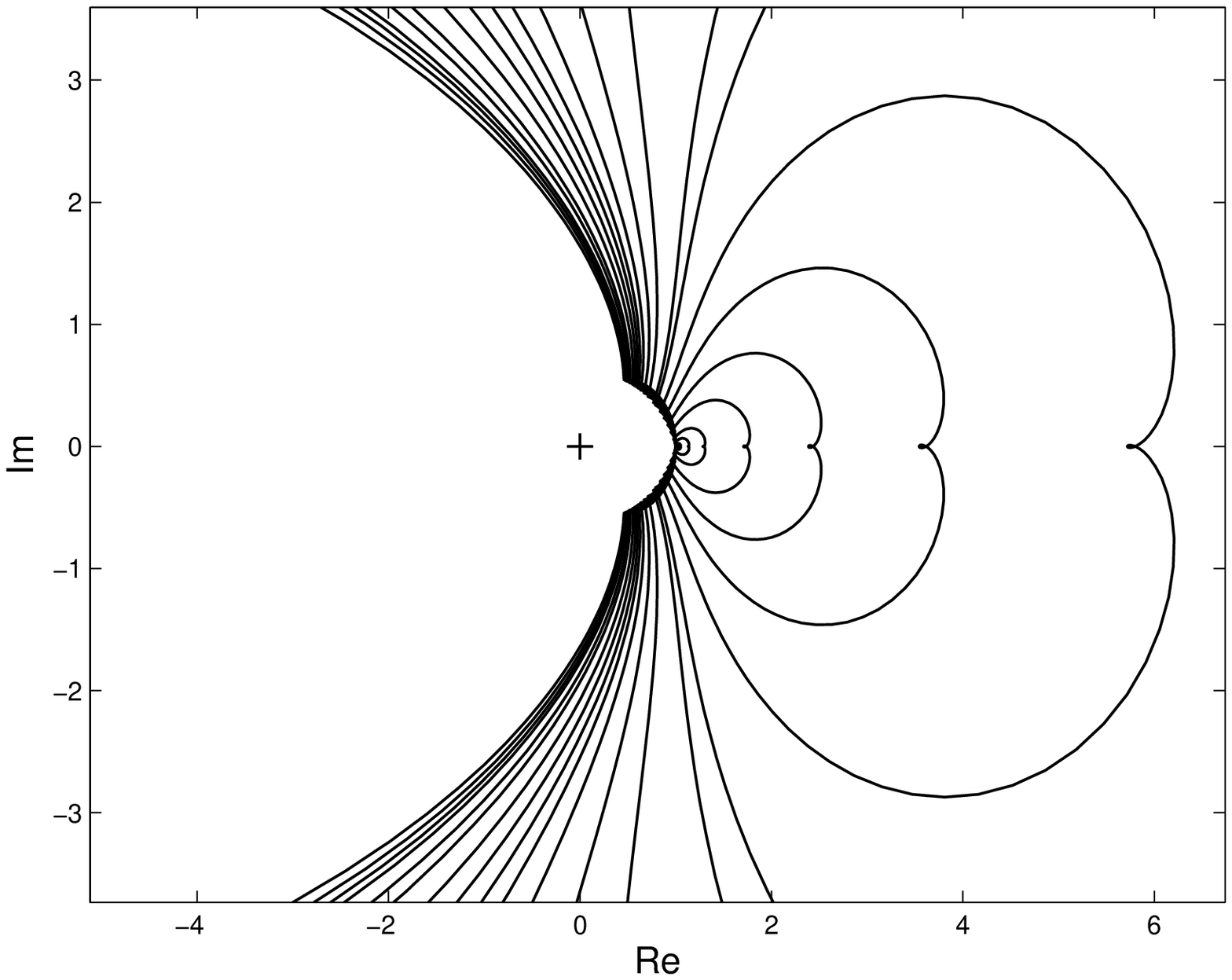} \\
\mbox{\bf (a)} & \mbox{\bf (b)}
\end{array}$
\end{center}
\caption{$\mathcal{E}_A = D = 1$ with $q$ varying through its admissible range with $k=1/8$.  These contours all have winding number zero despite them wrapping a little about the origin.}
\label{fig:smallK2}
\end{figure}

\subsubsection{The Majda--Rosales Limit ($D\to 0$)}
Next, $k=1=\mathcal{E}_A$ and $D\to 0$ (Majda-Rosales limit) is also interesting.  In this case, the tail does not lengthen as it does in the small $k$ and large $\mathcal{E}_A$ cases.  In the case of small $D$, we likewise find the Evans function contours do not get closer to the origin in the limit.  In fact, in Figure \ref{fig:smallD}, we see the contours getting farther away.  We remark that in the small $D$ limit, the high-frequency bounds blow up, and thus limited computational investigations are possible.

\begin{figure}[t]
\begin{center}
$\begin{array}{cc}
\includegraphics[width=7cm]{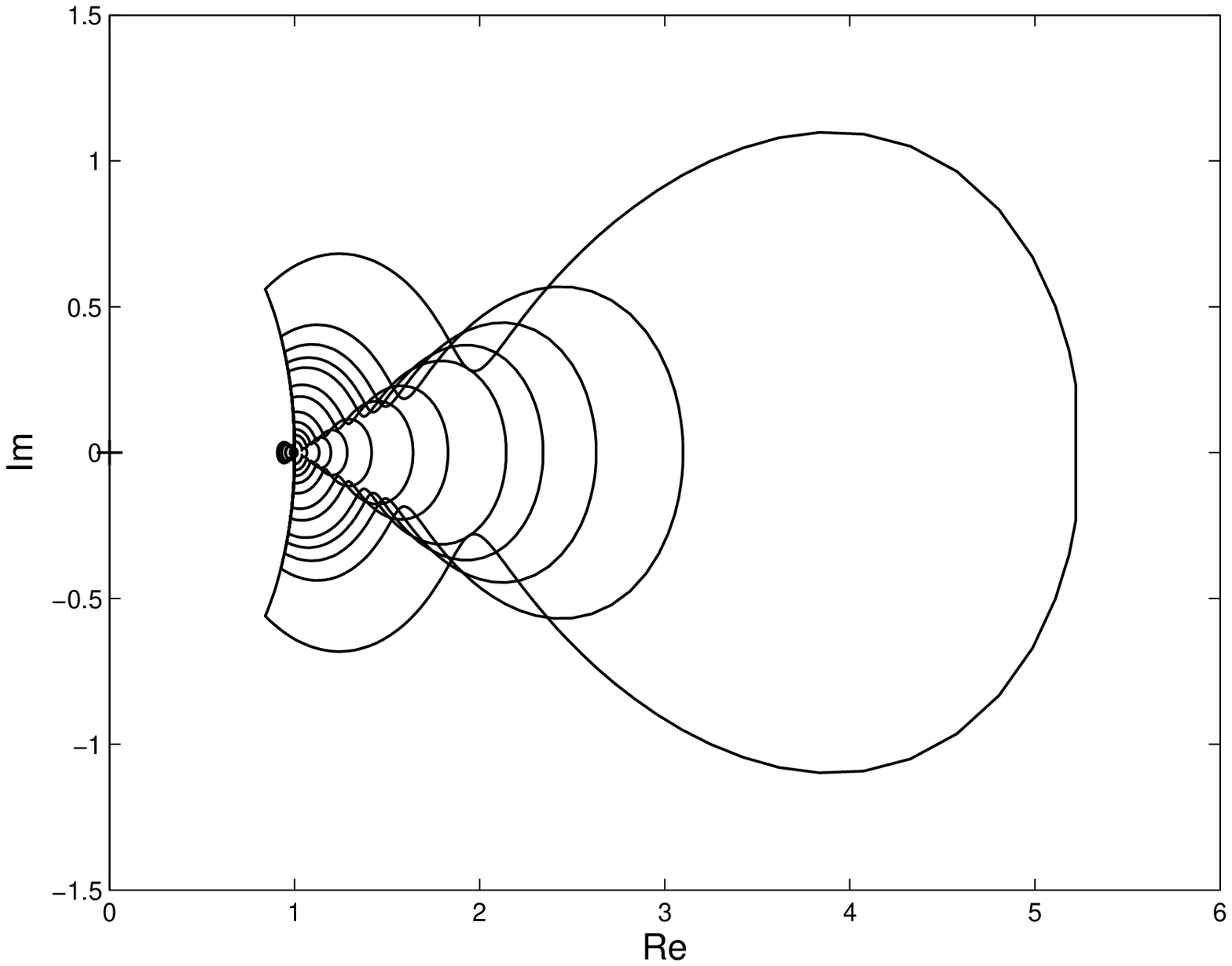} & \includegraphics[width=7cm]{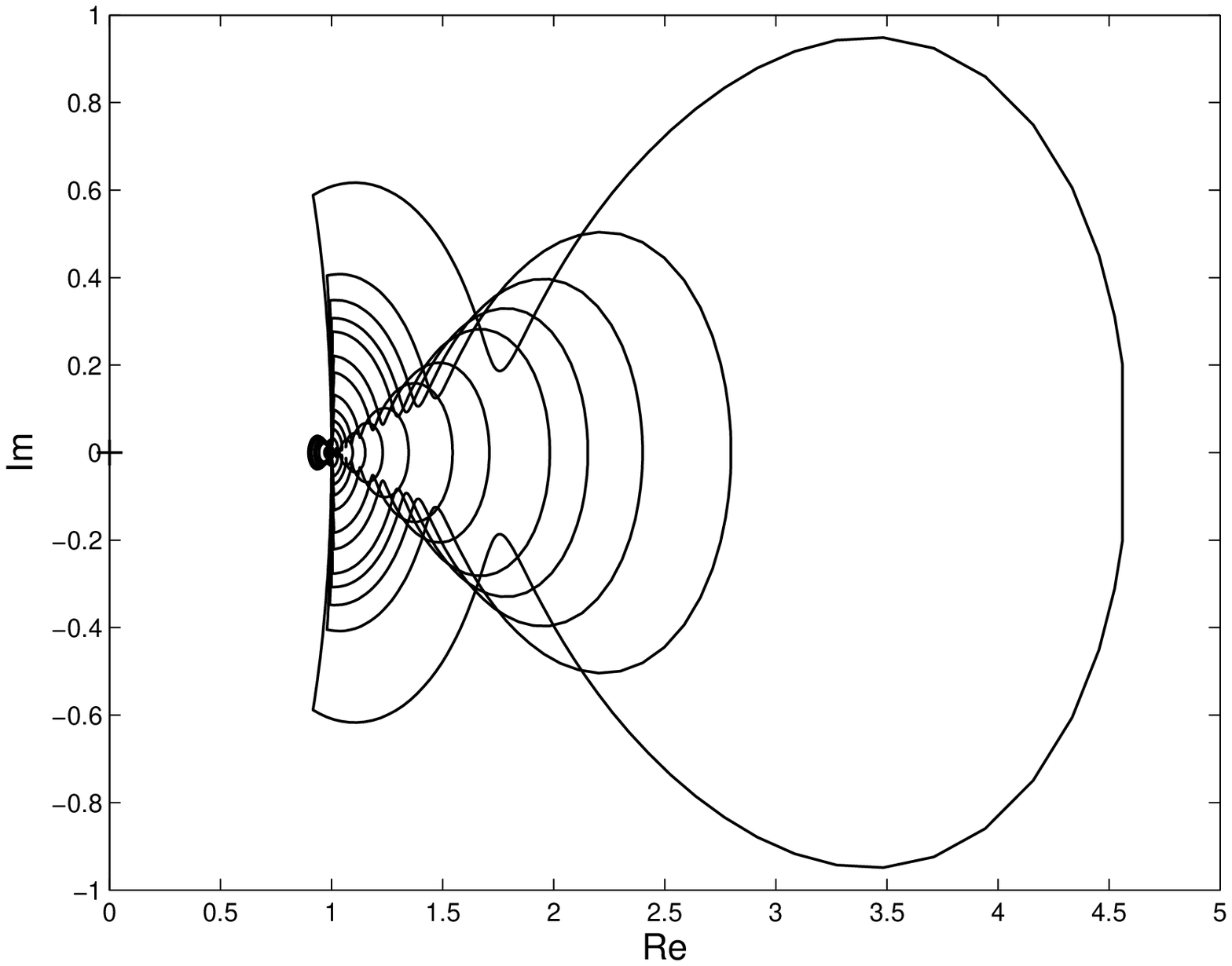} \\
\mbox{\bf (a)} & \mbox{\bf (b)}
\end{array}$
\end{center}
\caption{$\mathcal{E}_A = k = 1$ with $q$ varying through its admissible range with \textbf{(a)} $d=1/8$, $R=4$ and \textbf{(b)} $d=1/16$, $R=5.1$.}
\label{fig:smallD}
\end{figure}

\section{Conclusions}\label{sec:conclusions}
\subsection{Discussion}

\subsubsection{Findings}
The principal finding of our study is spectral stability. Combining the observed spectral stability with Proposition \ref{prop:lrtz}, we conclude that the detonation-wave solutions tested are \emph{nonlinearly} stable.  Indeed, we find no evidence of Hopf bifurcation as might be expected from the ``pulsating'' structures observed in the physical equations. As noted above this answers in the negative Majda's question as to whether his scalar combustion model might support such behaviors. 
Moreover, combining our work with that of Barker \& Zumbrun \cite{BZ_majda-znd} and the discussions in Kasimov et al.\ \cite{KFR} and Radulescu \& Tang \cite{RT_PRL11}, we may say with some confidence that the standard, scalar Majda/Fickett models---while providing good agreement with steady-state structures---do \emph{not} appear to reproduce the stability, or rather instability, properties of detonation waves that are seen in the physical systems. 

Rather than abandon the Majda/Fickett paradigm as too unphysical, Radulescu \& Tang \cite{RT_PRL11} have shown, by direct numerical simulation, that adding a certain forcing term to Fickett's model allows them to find pulsating detonation waves. 
Similarly, Kasimov et al.\ \cite{KFR} find that even the forced, scalar conservation law with viscosity,
\[
u_t+\frac{1}{2}(u^2-uu_s)_x=f(x,u_s)\,,
\]
considered on the quarter plane $(x,t)\in(-\infty,0]\times[0,\infty)$ with $u_s(t):=u(0,t)$ and a prescribed forcing function $f$,
better captures many of the detonation phenomena of interest, namely, steady traveling-wave solutions, and their instability through a cascade of period-doubling bifurcations.
%
%
We note, however, that in contrast to the approaches of \cite{KFR} and \cite{RT_PRL11}---based on direct numerical simulation of the partial differential equation, our approach via the Evans function allows one, in principle, to get detailed information about the eigenstructure of the linearized equations. Indeed, the Green function bounds of Lyng et al.\ \cite{LRTZ_JDE07} used in the proof of nonlinear stability provide a great deal of information about the behavior of solutions of the linearized equations. Moreover, in the presence of unstable eigenvalues, the Evans-function framework allows one to track the eigenvalues as system parameters vary; this feature has proven to be particularly important in our preliminary efforts to apply the ideas and techniques of this paper to the physical setting \cite{BHLZ}. 

\subsubsection{Computational Effort}

The overall investigation consisted of computing thousands Evans function contours and represented an immense and tedious interactive computational effort.  Profiles needed to be routinely continued into each edge of parameter space, exploring the large and small limits of the spectrum in each of $D$, $k$, $\mathcal{E}_A$, $q\in[0,1/2]$, and $u_\mathrm{ig}$.  As the domains for large $\mathcal{E}_A$ and small $k$ produced slowly decaying profiles on the left side, continuation required the computational domain to be stretched out continually.  Also in the $\mathcal{E}_A\to 0$ limit, continuation became impossible for large $q$, suggesting a loss of existence in the limit.  All large $q$ profiles had a combustion spike, and thus needed to be continued from the small $q$ case.  Numerical investigations were performed by \textsc{STABLAB} \cite{STABLAB} on an 8-core Mac Pro in the \textsc{MATLAB} computing environment. See Appendix \ref{sec:parameter} for complete details on the parameter choices for the experiments. 

\subsubsection{Organizing Centers}
In Majda's original paper \cite{M_SIAMJAM81}, the coefficient of species diffusion, $D$ in \eqref{eq:mm}, is taken to be zero. One consequence of this assumption is that the resulting system of traveling-wave equations is planar. This feature is strongly used in Majda's proof of the existence of traveling-wave solutions. For example, he uses the Poincar\'e--Bendixson theorem in the proof.  By contrast, in our setting ($D\neq 0$), the system of traveling-wave equations \eqref{eq:tw} forms a three-dimensional dynamical system, and the resulting dynamics are not so simple to describe or visualize. 

\br
For example, we note that, although the statements of the results obtained by Larrouturou's \cite{L_NA85} in his analysis of the existence problem for traveling-wave solutions of \eqref{eq:mm-d} are effectively the same as those obtained by Majda in the original paper \cite{M_SIAMJAM81} for \eqref{eq:mm0}, the substance of Larrouturou's proof is strikingly different due primarily to the nonplanar phase space. In particular, Larrouturou uses Leray--Schauder degree theory to solve the problem on $[-\alpha,\alpha]$ and then obtains sufficient control of his truncated-domain solutions to let $\alpha\to\infty$. 
\er

But, there are at least two interesting limits/organizing centers that do give intuition. The first, evidently, is the $D\to 0$ (``Majda'') limit studied by Majda (as described above) in which the relevant profile equations reduce to a planar system. Another simplifying limit is the $\mathcal{E}_A\to 0$ (zero-activation-energy) limit, in which coupling
reduces to a single point where $u=u_\mathrm{ig}$. We give an outline of the construction of traveling-wave profiles in the zero-activation-energy limit in Appendix \ref{sec:zeroe}. Moreover, we show that the associated Evans function has additional structure in this case, and, in any case, can easily be computed in the \textsc{STABLAB} environment. In either case we find that there is a range of parameter values for which strong detonations exist, bounded by a codimension-one curve on which weak detonations exist, outside of which no connecting profile exists, of either type.
These two organizing centers coalesce at the interesting double limit where $\mathcal{E}_A=0$ and $D=0$. Notably, we find in this case a loss of regularity in the profiles. The $z$ profile is now only continuous with kink at $u=u_\mathrm{ig}$. Again we see both weak and strong detonations, as $k$ is varied with other parameters held fixed; the argument is similar to the $\mathcal{E}_A=0$ case described in Appendix \ref{sec:zeroe}.  
Also, as expected, the Evans function in this reduced setting features an additional reduction in complexity. 

Finally, we note that the triple limit $\mathcal{E}_A=0, D=0, B=0$ (zero-activation-energy limit of ZND) was studied recently by Jung \& Yao \cite{JY_QAM12}. Remarkably, for this reduced problem (for which there are no weak detonations) Jung \& Yao were able to  explicitly establish the nonvanishing of the Evans--Lopatinski\u\i\ determinant (the analogue of the Evans function for discontinuous ZND detonation waves; see \cite{JLW_IUMJ05} for details), and therefore rigorously prove spectral stability of detonations in the reduced system. By contrast, we do not see how to analytically establish stability in either of the cases identified above as an organizing center. That is, the Majda limit and the zero-$\mathcal{E}_A$ limit each seem to require numerical evaluation of the Evans function as pursued here for the full system \eqref{eq:mm}.


\subsection{Future directions}
\subsubsection{Navier--Stokes}
The most natural and significant direction for further study is to apply these techniques and their necessary extensions to the physical case of the reactive  Navier--Stokes equations \eqref{eq:rns}.
As noted above in \S\ref{ssec:overview}, this represents a substantial increase in complexity. However, it addresses a fundamental issue. 
Indeed, from the pioneering early work of Erpenbeck \cite{E_PF62} to the now classic study of Lee \& Stewart \cite{LS}, the standard approach to the study of detonation stability has been in the context of the (inviscid) reactive Euler equations.\footnote{This would be system \eqref{eq:rns} with $\mu$, $\kappa$, and $d$ all set to zero.} A pair of more recent surveys \cites{SK_JPP06,BS_ARFM07} confirms the central role that these equations continue to play in the stability theory and other aspects of the theory of detonation phenomena more generally. 
However, as noted by Majda \cite{M_book}, the approach---although quite successful in ideal gas dynamics---of neglecting second-order terms and using other information, e.g., an entropy condition, to incorporate the effect of the neglected diffusive terms on the solution seems to be much more delicate in the case of combustion systems. Thus, a fundamental issue is to quantify and clarify the role of these second-order terms. An Evans-function-based approach to this problem was initiated by Lyng \& Zumbrun \cite{LZ_ARMA04}, and this paper represents one key step in this ongoing program. 
Continuing this program, together with Barker, the authors have 
carried out corresponding Evans computations
for detonation-wave solutions of the full reactive 
Navier--Stokes equations \cite{BHLZ}, 
with promising initial results.
We expect that the Evans-function framework will allow us to refine our understanding of the role of viscosity in detonation stability. 
As additional motivation, we note that the recent parallel work of Romick, Aslam, \& Powers \cite{RAP_JFM12} suggests that there are definitely regimes in which the diffusive effects of viscosity, heat conductivity, and species diffusion play a non-negligible role in the dynamics. Their study, based on direct numerical simulations by finite differences, examines the evolution of ZND profiles under the reactive Navier--Stokes equations \eqref{eq:rns}. Notably, they found that the stability characteristics of these waves were altered in the presence of diffusive effects with the onset of instability
as $\mathcal{E}_A$ is increased delayed by $\approx10\%$. This is in agreement with our own experiments \cite{BHLZ}, which, moreover, show
much more dramatic discrepancies at the level of unstable eigenvalue distribution after the onset of instability.  (These, being related to bifurcation and/or exchange of stability, are related to more detailed aspects of behavior beyond simple instability which are often of great importance.)
Finally, we note that other natural variations to consider include the consideration of planar waves in several space dimensions and more realistic reaction schemes.

\subsubsection{Weak detonations}
Within the simplified framework of the Majda model, another promising direction of study concerns the stability of weak detonations---the rare, undercompressive detonation waves described in Remark \ref{rem:weak}. Notably, the nonlinear analysis of Lyng et al. \cite{LRTZ_JDE07} recorded in Proposition \ref{prop:lrtz} above treats such waves, and the study of their nonlinear stability is, as in the case of the strong detonation waves treated here, reduced to the analysis of the zero set of an appropriate Evans function. It is therefore accessible to the techniques used in this paper. 
Indeed, as noted above in \S\ref{ssec:stability_evans}, a preliminary Evans-function analysis \cite{LZ_PD04} of weak-detonation solutions of \eqref{eq:mm0}  suggests that a transition to instability, should it occur, must necessarily occur through a Hopf-type bifurcation. In contrast to the case considered here, we note that one challenge that immediately presents itself in the case of weak detonations is the scarcity of such waves. Weak detonations only exist for distinguished parameter values; as described in Remark \ref{rem:weak}, the existence of such a wave corresponds to the structurally unstable intersection of invariant manifolds. Thus, the first step in a stability analysis of weak detonations is the development of a robust system for locating them. For this purpose, the technique of inflating the phase space, as described by Beyn \cite{B_IMAJNA90}, seems to work well. An additional difficulty is related to the failure of  integrated coordinates to remove the eigenvalue at the origin; this is a topic of our current investigation.

\subsubsection{Ignition}
It has long been recognized---at least since the analysis of Roquejoffre \& Vila \cite{RV_AA98}---that the shape and structure of the ignition function $\phi$ plays a significant role in the dynamics and the analysis of these traveling-waves. As just one example, a recent manifestation of this phenomenon appears in the paper of Jung \& Yao \cite{JY_QAM12}; their analysis hinges on the fact that $\phi$ is assumed to be a Heaviside function. It would thus be of interest to better quantify the effect of the form of $\phi$ on the stability properties of these waves. A closely related issue is the effect of the location of the ignition temperature $u_\mathrm{ig}$. 
Also of possible interest is the incorporation of a ``bump-type'' ignition function of the kind proposed by Lyng \& Zumbrun \cite{LZ_PD04}. As a example, one might try
$\phi(u)= \me^{\mathcal{E}_A/T(u)}$, where 
\[
T(u)= (u-1.5)^2- (u_*-1.5)^2\,,
\] 
modeling the temperature-velocity relation from the geometry of the physical system. Here, $u_*$ is the state value for the von Neumann shock. Barker \& Zumbrun \cite{BZ_majda-znd} considered such an ignition function in the nearby setting of the Majda-ZND model, and they were able to find ``square-waves'' in the limit of large activation energy $\mathcal{E}_A$. These well known structures play an important role in the theory of detonation \cite{FickettDavis}. This finding aligns with the claim put forward by Lyng \& Zumbrun \cite{LZ_PD04} that the introduction of a bump allows the scalar model to better mimic the wave structure of \eqref{eq:rns}. Although we do not pursue it here, we conjecture that a similar result (existence of square waves in the high-activation energy limit) holds as well in the viscous setting. See Appendix C of \cite{Z_ARMA11} for further discussion of square waves and the Majda model. 

\subsubsection{Singular limits}
Finally, as one other possibility, we mention that the experiments in Section \ref{ssec:limit} are oriented toward with the singular limits $D\to0$ and $k\to 0$. Thus, the computational approach used here will eventually break down for sufficiently small parameter values. To address the true limiting behavior, a more reasonable approach would be to attack these problems means of an Evans-function analysis involving the formal limiting problem as in \cite{Z_ARMA11}. 


\section*{Acknowledgement.}
We thank Blake Barker for his contributions through work on the concurrent
joint project \cite{BHLZ}, and in the ongoing development of the \textsc{STABLAB} package
with with these computations were carried out.

\appendix
\section{Zero activation energy: details}\label{sec:zeroe}
In this appendix, we provide more details about profiles and the Evans function in the case that $\mathcal{E}_A=0$. This is one of the organizing centers described above. 
Thus, we assume that the ignition function $\phi$ is of Heaviside type
(the formal limit as $\mathcal{E}_A\to 0$ of the Arrhenius version). 
We suppose
\[
\phi(u)=
\begin{cases}
1\,, & \text{if}\;u>u_\mathrm{ig} \\
0\,, & \text{if}\;u<u_\mathrm{ig}
\end{cases}\,.
\]

\subsection{Existence and structure of traveling-wave profiles}
We recall, from \eqref{eq:tw}, that the profile equations are 
\begin{align*}
\hat u'&=B^{-1}\big((\hat u^2/2)-(u^2_\sm/2)-s(\hat u-u_\sm)-q(s\hat z+D\hat y)\big)\,, \\
\hat z'&=\hat y\,,\\
\hat y'&=D^{-1}\big(-s\hat y+k\phi(\hat u)\hat z\big)\,.
\end{align*}
But, in the case that $\phi$ is a step function, the equation for the $\hat z$ component can be written as a second-order piecewise constant-coefficient linear equation. 
\beq
D\hat z''+s\hat z'-k\phi(\hat u)\hat z=0\,.
\eeq
There are two cases: $\hat u<u_\mathrm{ig}$ and $\hat u>u_\mathrm{ig}$. 
\subsubsection{Case 1: Below ignition $(\hat u<u_\mathrm{ig})$}
In this case, $\phi(\hat u)=0$, so the equation for $\hat z$ reduces to 
\beq
(Dz'+sz)'=0\,.
\eeq
Thus, the solution must satisfy 
\(
(\hat z')'=-\ti s\hat z'
\), 
where $\ti s:=sD^{-1}$, so that for some constant $c_0$, $z'$ is given by
\[
\hat z'(x)=\me^{-\ti sx}c_0
\]
Integrating one more time, we find 
\[
z_+-\hat z(x)=
\int_x^{+\infty} \hat z'(\xi)\,\dif\xi
=\int_x^{+\infty}\exp(-\ti s\xi)c_0\,\dif\xi
=-\frac{\me^{-\ti sx}c_0}{-\ti s}\,.
\]
That is,
\beq
\hat z(x)=1-\frac{\me^{-\ti sx}c_0}{\ti s}\,.
\eeq
\begin{remark}
We note that, as required, $\lim_{x\to+\infty}\hat z(x)=1$. However, one curiosity of this solution is that the $\hat z$ profile takes (assuming, as we'll see below, that $c_0>0$) values strictly smaller than 1 in the unburned zone ($u<u_\mathrm{ig}$). This is odd because $\phi$ vanishes in this zone and one expects that there should be no consumption of reactant in this region. 
\end{remark}

\subsubsection{Case 2: The burning tail $(u>u_\mathrm{ig})$}
In this case $\phi(\hat u)\equiv1$, and the $\hat z$-equation is 
\beq
D\hat z''+s\hat z'-k\hat z=0\,.
\eeq
Evidently, the characteristic polynomial is 
\(
D\mu^2+s\mu-k=0
\)
which has roots 
\[
\mu_\spm=\frac{-s\pm\sqrt{s^2+4Dk}}{2D}\,.
\]
In our framework, the parameters $s$, $k$, and $D$ are all positive. In order for there to be a connection, we require $\hat z$ to decay to $0$ as $x\to-\infty$ in the burning zone, so only the positive growth rate, $\mu_\sp$,
is of interest. Thus, the solution in the burned tail takes the form
\(
\hat z(x)=d_0\me^{\mu_\sp x},
\)
where $d_0$ is a constant of integration. The constant $d_0$ needs to be chosen so that the solution matches the right-hand solution.
That is, if $x_\mathrm{ig}$ is the point on the $x$-axis defined by $\hat u(x_\mathrm{ig})=u_\mathrm{ig}$, then $d_0$ needs to be chosen so that 
\beq\label{eq:match-z}
d_0\me^{\mu_\sp x_\mathrm{ig}}=1-\frac{\me^{-\ti sx_\mathrm{ig}}c_0}{\ti s}\,.
\eeq

\begin{remark}\label{rem:xi}
We use translational invariance to take, without loss of generality, $x_\mathrm{ig}=0$.
\end{remark}
Using Remark \ref{rem:xi}, we see immediately that 
\eqref{eq:match-z} reduces to 
\beq\label{eq:match-z2}
d_0=1-\frac{c_0}{\ti s}\,.
\eeq
We also want to match the derivative $\hat z'$ at $x_\mathrm{ig}=0$. This requires 
\beq\label{eq:match-zd}
\mu_\sp d_0=c_0\,.
\eeq
Thus, combining \eqref{eq:match-z2} and \eqref{eq:match-zd}, we find
\beq
\begin{pmatrix} 
1 & \ti s^{-1} \\
\mu_\sp & -1
\end{pmatrix}
\begin{pmatrix} d_0 \\ c_0 \end{pmatrix}
=
\begin{pmatrix}
1 \\
0
\end{pmatrix}
\eeq
The determinant of the coefficient matrix is $-1-\mu_\sp\ti s^{-1}\neq 0$, whence we obtain
\[
d_0=\frac{1}{1+\mu_\sp\ti s^{-1}}\,,\quad
c_0=\frac{\mu_\sp}{1+\mu_\sp\ti s^{-1}}\,.
\]

\subsubsection{The equation for $\hat u$}
The equation for $\hat u$ now takes the form of the viscous profile equation for the Burgers equation with a piecewise defined forcing term:
\beq\label{eq:u}
B\hat u'=\left(\frac{\hat u^2}{2}-s\hat u\right)-\left(\frac{u_\sm^2}{2}-su_\sm\right)+
\begin{cases}
-qd_0\me^{\mu_\sp x}(s+D\mu_\sp) & x<x_\mathrm{ig} \\
-qs & x>x_\mathrm{ig}
\end{cases}\,.
\eeq

We observe that 
\(
d_0(s+D\mu_\sp)=s
\),
and thus there is no discontinuity in the forcing term at $x_\mathrm{ig}=0$ on the right-hand side in equation \eqref{eq:u}, and we do not expect $\hat u$ to have a kink at $x_\mathrm{ig}$. That is, $u'(x_\mathrm{ig}+)= u'(x_\mathrm{ig}-)$. This is consistent with the fact that we can alternatively write the profile equation for $\hat u$ as 
\[
B\hat u'=\left(\frac{\hat u^2}{2}-s\hat u\right)-\left(\frac{u_\sm^2}{2}-su_\sm\right)-q(s\hat z+D\hat z')\,,
\]
and we note that for $z\in C^1(\RR)$, as shown above, the right-hand side of the above ODE is perfectly continuous.

\subsubsection{Numerics for profile}
Since the form of $\hat z$ is given explicitly, it is a simple matter to compute the profile $\hat u$ simply by integrating \eqref{eq:u} forward and backward in $x$ from $x=0$. Figure \ref{fig:profileA} shows three profiles computed in this way. In the figure, the value of $q$ has been tuned to show the ``bench'' phenomenon at height $u_\sm^\mathrm{w}$. 
\begin{figure}[ht] %
   \centering
   \begin{tabular}{ccc}
   \includegraphics[width=5cm]{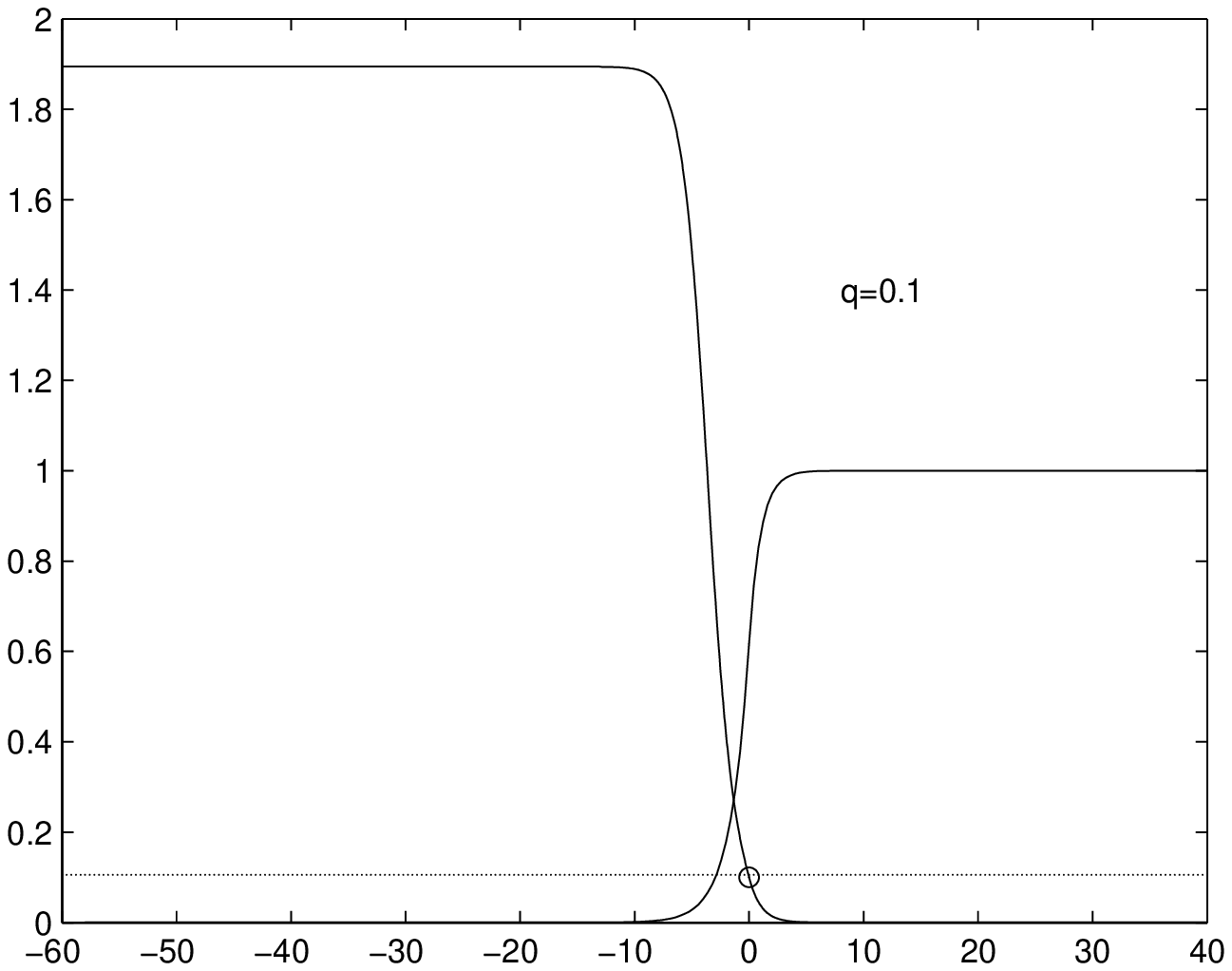} &
   \includegraphics[width=5cm]{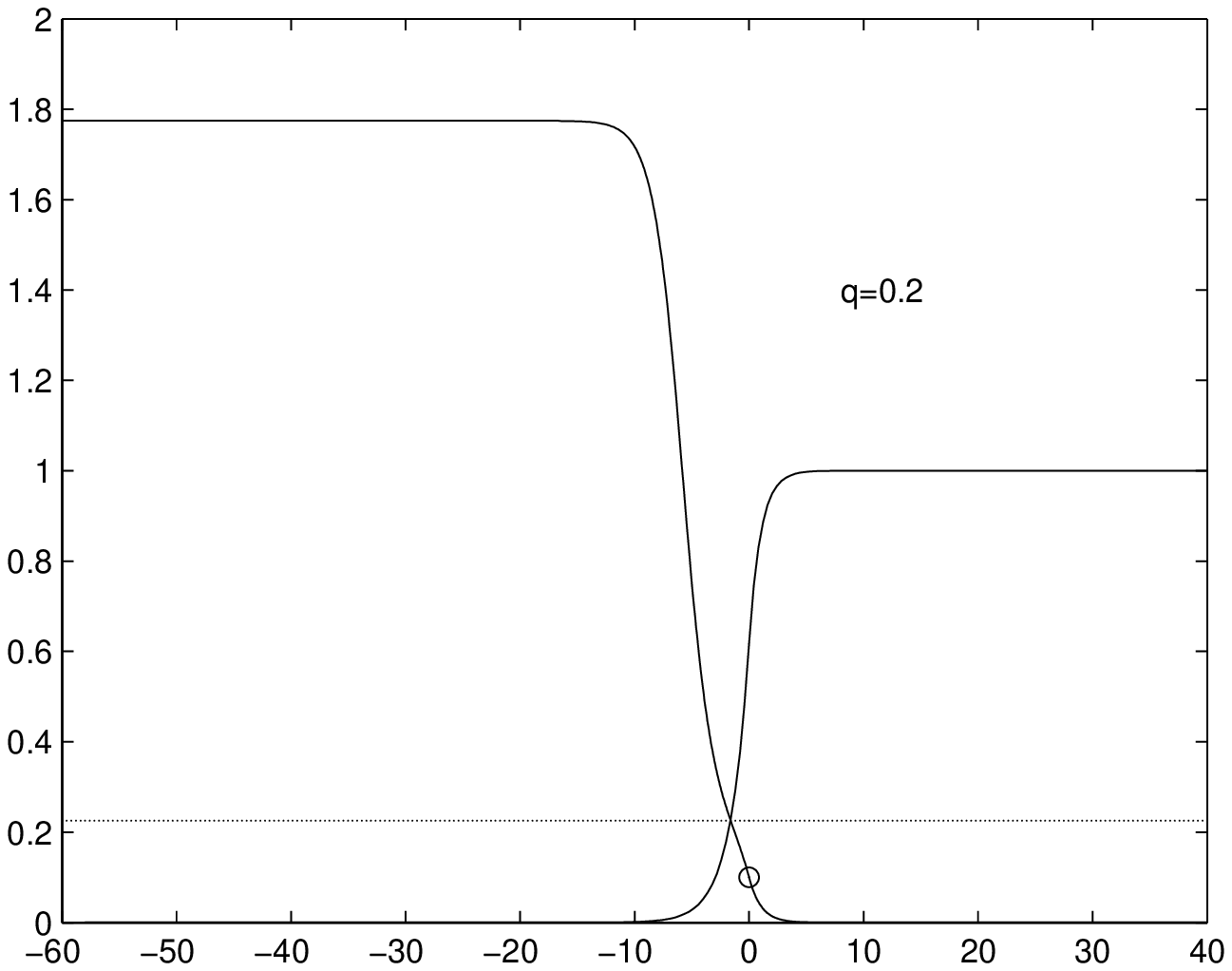} &
   \includegraphics[width=5cm]{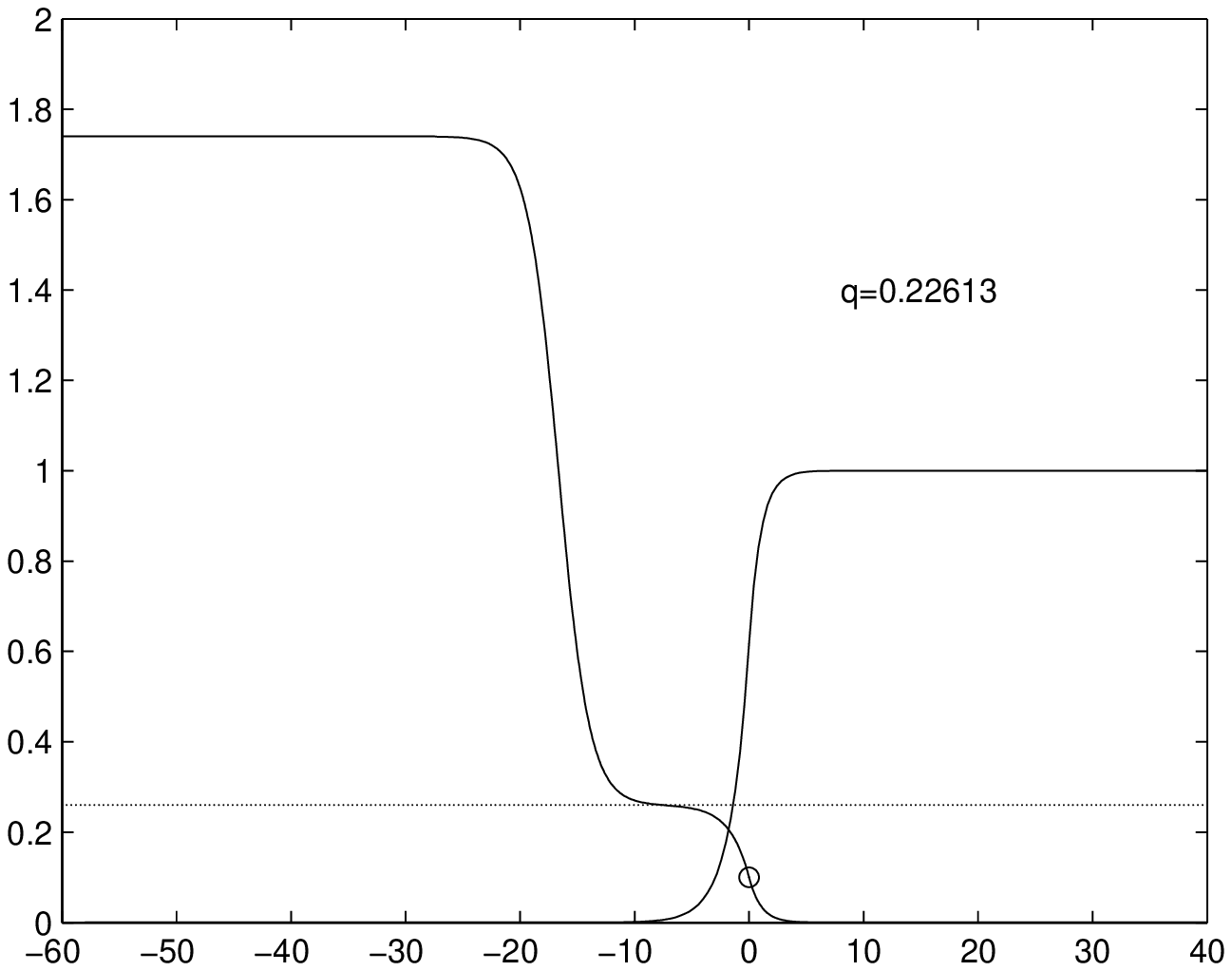} \\
   \textbf{(a)} & \textbf{(b)} & \textbf{(c)}
   \end{tabular}
   \caption{``Benching'': Solutions $\hat u$ of \eqref{eq:u} plotted together with the explicitly known functions $\hat z$ for three different values of $q$. As in Figure \ref{fig:profiles}, the two curves can be distinguished by their boundary behavior. The function $\hat z$ tends to unity on the right and to zero on the left. In all three panels, $u_\sp=0$, $u_\mathrm{ig}=0.1$, $k=D=1$. The value of $u_\sm^\mathrm{w}$ is marked by the faint dotted horizontal line. The initial condition at $x_\mathrm{ig}=0$ is marked by a small circle. The $q$ values are \textbf{(a)} $q=0.1$, \textbf{(b)} $q=0.2$, and \textbf{(c)} $q=0.226125505$. }
   \label{fig:profileA}
\end{figure}

\subsection{Evans Function}
%
\subsubsection{First-order formulation}
Since $\phi$ is a Heaviside function, linearization about the steady solution $(\hat u,\hat z)$ introduces a Dirac $\delta$ into the equations. 
We see immediately that the linearized system can be written as 
\begin{subequations}
\label{eq:majda-lin2}
\begin{align}
\lam u - su' + \left(\hat u u\right)' &= Bu'' + q k H_{u_{\mathrm{ig}}}(\hat u) z+qk\delta_{u_{\mathrm{ig}}}(\hat u)u\hat z\,,\label{majda-lin:a2}\\
\lam z - sz' &= D z'' - k H_{u_{\mathrm{ig}}}(\hat u) z-k\delta_{u_{\mathrm{ig}}}(\hat u)u\hat z. \label{majda-lin:b2}
\end{align}
\end{subequations}
We define $w':=u+qz$ so that \eqref{majda-lin:a2} can be written as 
\begin{equation}
Bu''=\lambda w'-qz' -qDz''+((\hat u-1) u)'.\label{eq:int1}
\end{equation}
which can be integrated so that the eigenvalue equation becomes
\begin{subequations}\label{eq:intevalsystemA}
\begin{align}
u'&=B^{-1}\big(\lambda w-qz-qDz'+(\hat u-s) u\big)\,, \\
w'&=u+qz\,,\\
z''&=D^{-1}\big(\lambda z-z'+k\delta_{u_{\mathrm{ig}}}(\hat u)u\hat z+kH_{u_{\mathrm{ig}}}(\hat u)z\big)\,.
\end{align}
\end{subequations}
In matrix form with unknown $X:=(U,Z)^\tr=(u,w,z,z')^\tr$, \eqref{eq:intevalsystem} takes the form 
\beq
\label{eq:int_eval_odeA}
X'=\mathbb{B}(x;\lambda) X
\eeq
with
\begin{align}
\mathbb{B}(x;\lambda)&:=\left[\begin{array}{cc|cc}
B^{-1}(\hat u-s) & \lambda B^{-1} & -qB^{-1} & -qB^{-1}D \\
1 & 0 & q & 0 \\ \hline
0 & 0 &0 & 1\\
D^{-1}k\delta_{u_{\mathrm{ig}}}(\hat u)\hat z & 0 & D^{-1}(\lambda+kH_{u_{\mathrm{ig}}}(\hat u)) & -D^{-1}
\end{array}\right] \nonumber \\
&=:\left[\begin{array}{c|c}
\mat{A}(x;\lam) & \mat{B}(x;\lam) \\ \hline
\mat{C}(x;\lam) & \mat{D}(x;\lam)
\end{array}
\right]\,.
\end{align}
The limiting, constant-coefficient system $X'=\mathbb{B}_\spm(\lambda)X$ is also of interest, and we see that the limiting blocks $\mat{C}_\spm(\lam)$ are both zero. Thus, the limiting systems are upper block-triangular.  However, this feature can be found also in the limiting systems associated with the coefficient \eqref{eq:bmatrix} above. What is remarkable in the current scenario is that, that if $\hat u$ satisfies $\hat u(x)>u_{\mathrm{ig}}$ for $x<x_\mathrm{ig}=0$ and $\hat u(x)<u_{\mathrm{ig}}$ for $x>x_\mathrm{ig}=0$, then the coefficient matrix $\mathbb{B}$ simplifies considerably on each half line. That is, the coupling due to the lower left-hand block $\mat{C}$ is concentrated in a single point. Moreover, the lower right-hand block $\mat{D}$ is constant on each half line. Thus, in this case, eigenvectors corresponding to the eigenvalues from $\mat{D}$ can be computed explicitly, and other decaying solutions come from a forced linear system of the form $U'=\mat{A}U+\mat{B}Z$. Finally, in the construction of the Evans function, the solutions must be ``matched'' to account for the point source at the origin. 

\begin{remark}
All the profiles $\hat u$ that we have computed satisfy the above inequalities. We note also that even more structure is available in the coefficient matrix in the original ``unintegrated'' coordinates. In that case the system becomes block diagonal on the right, whence the construction of decaying solutions at $x=+\infty$ decouples into pure ``fluid'' and ``reaction'' modes. However, in either case, we are not quite able to get our hands on a useful analytic expression for the whole Evans function. Therefore, we compute values of $E(\lam)$ using \textsc{STABLAB} as above. Because of this, we omit the detailed (but partial and ultimately not-so-useful) calculations that come from a careful analysis of the matrix $\mathbb{B}$ and its blocks. 
\end{remark}

\subsubsection{Coupling at $x=0$ and computation of $E(\lambda)$}
To compute $E$, we must account for the coupling at the origin. 
Basically, we have an ODE of the form
\[
X'=\delta(x)M(x)X\,,\quad X(0^\sp)=X_0\,,
\]
with $M$ continuous. The basic task is to determine $X(0^\sm)$ so that an Evans function can be formed. One can visualize this process as pulling solutions decaying at $+\infty$ across the point source at the origin so that they can be used in an Evans function (together with a basis of solutions decaying at $-\infty$).  
We observe that 
\begin{align*}
X(0^\sm)& =\me^{\int_{0^\sp}^{0^\sm}\delta(x)M(x)\,\dif x}X(0^\sp) \\
	&=\me^{-M(0)}X_0
\end{align*}
%
\begin{remark}
We have used that 
\[
M(x)=M(0)+O(|x|) \quad\text{for $x$ small}\,,
\]
so that we can exponentiate a constant matrix in the ODE calculation above. 
\end{remark}
%
Thus, we define 
\[
\calZ=\me^{-M(0)}X
\] 
and compute that 
\beq
\calZ'=\me^{-M(0)}\mathbb{B}(\me^{-M(0)})^{-1}\calZ\,.
\eeq
This gives a piecewise defined coefficient matrix
\beq
\tilde{\mathbb{B}}=\me^{-M(0)}\mathbb{B}(\me^{-M(0)})^{-1}
\eeq
that can easily be used in the standard implementation of \textsc{STABLAB}; see Figure \ref{fig:E0}.
\begin{figure}[ht]
\begin{center}
$\begin{array}{cc}
\includegraphics[width=7cm]{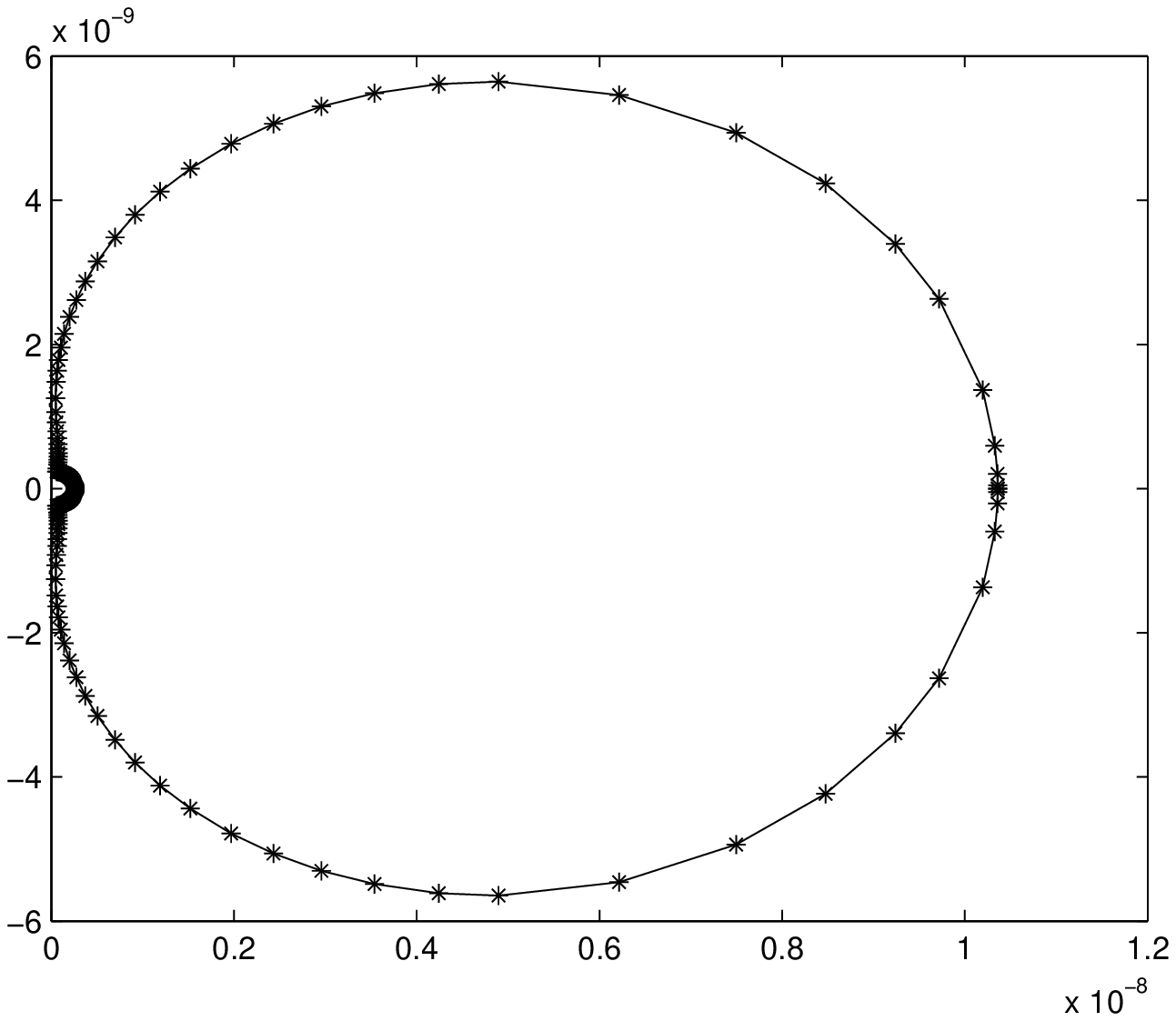} & \includegraphics[width=7cm]{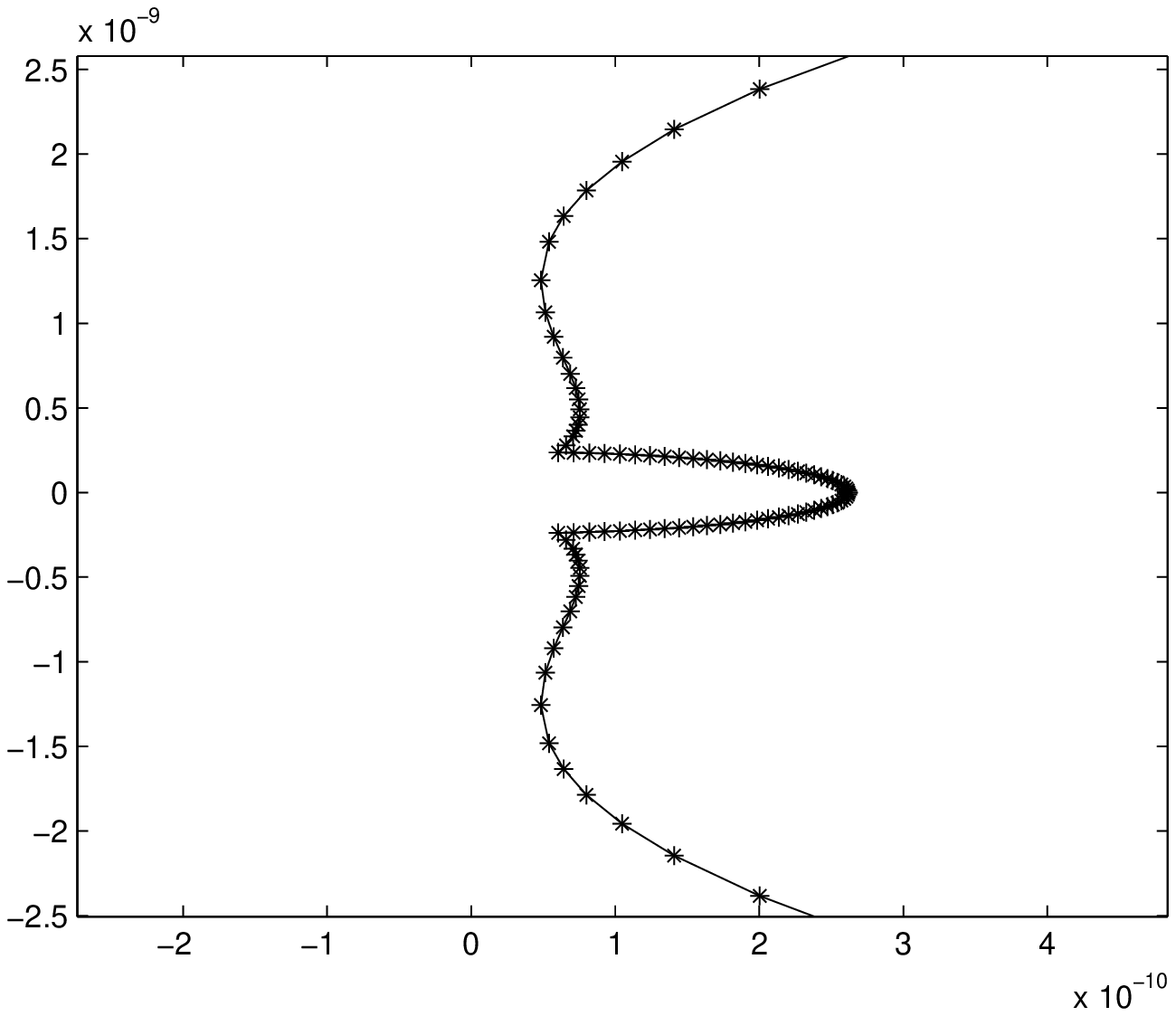} \\
\mbox{\bf (a)} & \mbox{\bf (b)}
\end{array}$
\end{center}
\caption{Sample Evans output for $\mathcal{E}_A=0$. \textbf{(a)}: the unscaled Evans-function output of a semicircle of radius $10$ computed using standard \textsc{STABLAB} routines for the profile with $q=0.1$ shown in Figure \ref{fig:profileA}\textbf{(a)}; \textbf{(b)}: a magnification of the origin.}
\label{fig:E0}
\end{figure}

\section{Parameter Values}\label{sec:parameter}

For the comprehensive study, we explored, with some exceptions, the following parameters:
\begin{align*}
(\mathcal{E}_A,k,D,u_+,u_\mathrm{ig})&\in\{0.01, 0.05, 0.1, 0.125, 0.25, 0.5, 1, 2, 4\}\\
&\quad\times \{0.125, 0.25, 0.5, 1, 2, 4\} \times \{0.0625, 0.125, 0.25, 0.5, 1, 2, 4\}\\
&\quad\times \{0, 0.1, 0.2, 0.3, 0.4, 0.6\} \times \{0.01, 0.1, 0.2\},
\end{align*}
with $q$ varying between zero and $q_{\mathrm{max}}$, which is determined by \eqref{eq:uq}.  For example for $u_+=0$, we used the following $25$ values
\begin{equation}
\label{eqn:q_in_large_E}
\begin{split}
q&\in\{0.001, 0.005, 0.01, 0.025, 0.05, 0.1, 0.125, 0.2, 0.25, 0.3, 0.32,\\
&\qquad 0.37,0.40,0.425,0.44,0.45,0.46,0.47,0.48,0.485,0.49,0.492,0.494,0.496,0.499\}
\end{split}
\end{equation}
and chose a similar range for the other values of $u_\sp$.  We note that roughly $2.7\%$ of the above parameter combinations yielded profiles with very slow exponential decay, and thus we not practical for computation.  This occurs, for example, when $k$ is small, $\mathcal{E}_A$ is large, and $q$ is close to $q_{\mathrm{max}}$ (all three).

\begin{remark}
We note that the impractical parameter combinations described above correspond to the inviscid limit.  The stability of this limiting case was examined in \cites{Z_ARMA11,BZ_majda-znd}.
\end{remark}

For the high activation energy study, we set $k=\exp(\mathcal{E}_A/2)$ for
\[
\mathcal{E}_A = \{1,2,4,6,8,10,12,14,16,18,20,22,24,26,28,30,32,34,36,38,40,42,44\}.
\]
The other parameters were set to $D=1$, $u_+=0$, and $u_{ig}=0.1$, with $q$ in \eqref{eqn:q_in_large_E}.


\end{document}